\documentclass[12pt,a4paper]{amsart}
\usepackage{amssymb,amscd}
\usepackage{amsfonts}
\usepackage[top=35mm, bottom=35mm, left=30mm, right=30mm]{geometry}
\usepackage[colorlinks=true,citecolor=blue]{hyperref}
\usepackage{mathptmx}
\usepackage{eucal}
\usepackage{graphicx}
\usepackage{mathrsfs}
\usepackage{amssymb}
\usepackage{amsmath}
\usepackage{amsthm}
\usepackage{xcolor}
\usepackage[pagewise]{lineno}\nolinenumbers

\newtheorem{theorem}{Theorem}[section]
\newtheorem{proposition}[theorem]{Proposition}
\newtheorem{lemma}[theorem]{Lemma}
\newtheorem{corollary}[theorem]{Corollary}

\newtheorem*{open question}{Open Question}

\newtheorem*{claim}{Claim}

\theoremstyle{definition}
\newtheorem{definition}[theorem]{Definition}

\numberwithin{equation}{section}

\makeatletter

\newcommand{\Rmnum}[1]{\expandafter\@slowromancap\romannumeral #1@}
\makeatother

\begin{document}
\title{Katok's intermediate entropy conjecture for amenable group actions}
\author{Yage Liu, Ercai Chen and Xiaoyao Zhou*}
\address
{1. School of Mathematical Sciences,\small Ministry of Education Key Laboratory of NSLSCS\\	\small  Nanjing Normal University, Nanjing 210023, Jiangsu, P. R. China}
\email{liuyage16@163.com }
\email{ecchen@njnu.edu.cn}
\email{zhouxiaoyaodeyouxian@126.com}
\renewcommand{\thefootnote}{}
\footnotetext{*Corresponding author: zhouxiaoyaodeyouxian@126.com}                 
\renewcommand{\thefootnote}{\arabic{footnote}}
\begin{abstract}
In this paper, we investigate a broad class of systems for
which Katok's intermediate entropy conjecture holds. In
particular, we show that a dynamical system with an amenable
group action satisfies Katok's intermediate entropy conjecture
provided that it has the specification property and is
asymptotically entropy expansive.
\end{abstract}
\maketitle
\section{Introduction}
Whether positive topological entropy implies a rich structure of the space of invariant measures is a long-standing problem. Parry asked if every strictly ergodic system must have zero entropy. $C^0$ counterexamples (e.g.~\cite{BCL07,GW94,HK67}) show it is not always true. However, one may expect a positive answer for smooth systems,
as conjectured by Herman, since in this setting positive
topological entropy implies the existence of nonzero Lyapunov
exponents, from which some hyperbolic structure may be obtained. Katok~\cite{K80} proved for $C^{1+\alpha}$ surface diffeomorphisms that positive entropy implies horseshoes, hence ergodic measures with all intermediate metric entropies, and conjectured the same for any sufficiently regular smooth system in any dimension. 

\textbf{Katok's intermediate entropy conjecture} (Katok's conjecture for short) asserts that for any $C^2$ diffeomorphism $T$ on a compact Riemannian manifold $X$, the collection of ergodic entropies
$$\mathcal{H}(X, T) := \{ h_\mu(T) : \mu \in \mathcal{M}_e(X,T) \}$$
contains $[0, h(T))$.

Over the years, Katok's conjecture has been verified in a diverse range of dynamical settings. Sun's works~\cite{S10a,S10b,S12} made progress on the problem for
certain skew-product systems and linear toral automorphisms. In
particular, Sun~\cite{S12} established density of ergodic entropies
for all linear toral automorphisms. A breakthrough came when Quas and Soo~\cite{QS16} proved that any asymptotically entropy expansive system enjoying the almost weak specification property together with the small boundary property is universal. The intermediate entropy property then follows directly.
Subsequent work by Burguet~\cite{B20} and Chandgotia and
Meyerovitch~\cite{CM21} further developed this universality
approach in broader settings, including systems satisfying
weaker specification-type assumptions. Working independently, Guan, Sun and Wu~\cite{GSW17} showed that certain homogeneous systems possess the almost weak specification property, thus confirming Katok’s conjecture within the universality framework.
It should be noted that the approximate product property alone is too weak to guarantee universality. The ideas of Quas and Soo also played an important role in the work of Huang, Xu and Xu~\cite{HXX21}, who established the
conjecture for affine transformations of nilmanifolds having a
periodic point and, in particular, for quasi-hyperbolic affine
transformations.

Alongside the developments discussed above, researchers have also taken different paths. Ures~\cite{U12} succeeded in proving Katok's conjecture for a specific class of partially hyperbolic diffeomorphisms whose center is one-dimensional. Yang and Zhang~\cite{YZ20} studied
a broad class of robustly non-hyperbolic transitive
diffeomorphisms and showed that every ergodic measure can be
approximated, both in the weak-$*$ topology and in entropy, by
ergodic measures supported on hyperbolic sets, thereby
establishing the conjecture. Konieczny, Kupsa and Kwietniak~\cite{KKK18} treated shift spaces equipped with a safe symbol (all hereditary shifts fall into this category), showing that the set of ergodic invariant measures becomes arcwise connected under the $d$-bar metric, which suffices to prove the conjecture. Li and Oprocha~\cite{LO18} proved Katok's conjecture for
topologically transitive systems with the shadowing property
and upper semi-continuous entropy maps. Their argument is based on approximations by ergodic measures
supported on odometers and on almost one-to-one extensions of
odometers. In a more recent contribution, Sun~\cite{S21} introduced a novel technique relying on the uniqueness of equilibrium states, thus proving the conjecture for a class of Mañé systems. In a later paper~\cite{S25}, Sun offered yet another confirmation of Katok's conjecture by revealing a refined structural property of the invariant measure space: ergodic measures of intermediate entropies and of intermediate pressures are generic in their respective subspaces. More recently, the authors~\cite{LCZ26} established the existence of ergodic
measures of intermediate entropies for $\mathbb Z^d$-actions.

A natural question is whether Katok's conjecture holds for
amenable group actions. To formulate our results, let $G$ be
an infinite countable discrete amenable group. A
$G$-topological dynamical system ($G$-tds for short) $(X,G)$
consists of a compact metric space $X$ (with metric $\rho$) and
a continuous action of $G$ on $X$, i.e., a map
$G\times X\to X$, $(g,x)\mapsto gx$, satisfying
$ex=x$ and $g(hx)=(gh)x$ for all $g,h\in G$ and $x\in X$,
where $e$ is the identity element of $G$. Let
$\mathcal M(X)$, $\mathcal M(X,G)$, and $\mathcal M_e(X,G)$
denote the sets of Borel probability measures on $X$,
$G$-invariant measures, and $G$-invariant ergodic measures,
respectively. Let $h(X,G)$ denote the topological entropy, and,
for $\mu\in\mathcal M(X,G)$, let $h_\mu(X,G)$ denote the
entropy of $\mu$.

A dynamical system is said to have the \textit{intermediate
entropy property} if it satisfies Katok's conjecture. In this
paper, we give an affirmative answer to the above question for
asymptotically entropy expansive amenable group actions with
the specification property. The main step in the proof is the
construction of compact invariant sets whose entropy is
controlled near a prescribed level. We carry out this
construction using zero-entropy exact tilings and employ
symbolic dynamics as a coding tool
(see Proposition~\ref{pro:3.1} and~\cite{DHZ19}).
Furthermore, by combining this construction with asymptotic
entropy expansiveness, we establish the intermediate entropy
property. Our main theorem reads as follows.
\begin{theorem}\label{thm:1.1}
Let $(X,G)$ be an asymptotically entropy expansive system with
the specification property. Then
\[
\mathcal H(X,G)
:=
\left\{
h_\mu(X,G):
\mu\in\mathcal M_e(X,G)
\right\}
=
[0,h(X,G)].
\]
\end{theorem}
A subset $\Lambda\subseteq X$ is $G$-invariant if for each $g\in G$, $g\Lambda=\Lambda$. For compact $G$-invariant $\Lambda$, $(\Lambda,G)$ is a topological dynamical system, so $\mathcal{M}(\Lambda,G)$ makes sense. Adapting the ideas of \emph{almost entropy-approximable},
\emph{entropy-approximable} and \emph{entropy-generic}
from \cite[Definition 1.2]{S25}, we define their counterparts
for continuous \(G\)-actions.
\begin{definition}
Let  $(X,G)$ be a topological dynamical system.\\
(1) Given $\mu\in\mathcal{M}\left( X,G\right)$, we say that $\mu$ is \textit{almost entropy-approximable} (by compact invariant sets of intermediate entropies) if for every neighborhood $U$ of $\mu$, every $h\in(0,h_{\mu}(X, G))$ and every $\varepsilon,\beta>0,$ there exist a compact $G$-invariant set $\Lambda$ and $\gamma\in(0,\varepsilon)$ such that 
$$\mathcal{M}(\Lambda,G)\subset U,~~h(\Lambda,G)>h~~and~~h(\Lambda,G,\gamma)<h+\beta,$$
where $h(\Lambda,G,\gamma)$ is the Bowen topological entropy of the subsystem $(\Lambda,G)$ calculated at the scale $\gamma$.\\
(2) Given $\mu\in\mathcal{M}\left( X,G\right)$, we say that $\mu$ is \textit{entropy-approximable} (by compact invariant sets of intermediate entropies) if for every neighborhood $U$ of $\mu$, every $h\in(0,h_{\mu}(X, G))$ and every $\beta>0,$ there is a compact $G$-invariant set $\Lambda$ such that 
$$\mathcal{M}\left(\Lambda,G\right)\subset U~~and~~h<h(\Lambda,G)<h+\beta.$$
(3) We say that the system $(X,G)$ is \textit{entropy-generic} if for every $\alpha\in[0,h(X, G))$, the set 
$$\mathcal{M}_{e}\left( X,G,\alpha\right):=\left\lbrace \mu\in\mathcal{M}_{e}\left( X,G\right): h_{\mu}(X, G)=\alpha\right\rbrace $$
is residual in the subspace
$$\mathcal{M}^{\alpha}\left( X,G\right):=\left\lbrace \mu\in\mathcal{M}\left( X,G\right): h_{\mu}(X, G)\geq\alpha\right\rbrace.$$
\end{definition}
When $h_\mu(X,G)=0$, almost entropy-approximability and
entropy-approximability are vacuous under the preceding
definitions. To avoid this degeneracy, we adopt the following
non-vacuous convention: we say that $\mu$ is almost
entropy-approximable (respectively, entropy-approximable) if,
for every neighborhood $U$ of $\mu$, there exists a nonempty
compact $G$-invariant set $\Lambda\subset X$ such that
$\mathcal M(\Lambda,G)\subset U$.

The following is our key theorem that requires only the specification property.
\begin{theorem}\label{thm:1.3}
Let $(X,G)$ be a system with the specification property. Then every $\mu\in \mathcal{M}\left( X,G\right)$ is almost entropy-approximable. 	
\end{theorem}
Under the additional assumption of asymptotic entropy expansiveness, we obtain a stronger conclusion that directly yields Theorem~\ref{thm:1.1}. Indeed, entropy-genericity is substantially stronger than the
intermediate entropy property. Moreover, asymptotic entropy
expansiveness guarantees the existence of at least one measure
of maximal entropy.
\begin{theorem}\label{thm:1.4}
	Let $(X,G)$ be an asymptotically entropy expansive system with the specification property. Then the following hold:
	\begin{enumerate}
		\item\label{theorem 4:1} Every invariant measure $\mu\in \mathcal{M}\left( X,G\right)$ is entropy-approximable.
		\item\label{theorem 4:2} The system $(X,G)$ is entropy-generic.
	\end{enumerate}
\end{theorem}
When $h(X,G)=0$, entropy-genericity holds vacuously.
Nevertheless, since the specification property implies
entropy-denseness, $\mathcal M_e(X,G)$ is dense in
$\mathcal M(X,G)$. Since $\mathcal M_e(X,G)$ is a
$G_\delta$ subset of $\mathcal M(X,G)$, it is residual in
$\mathcal M(X,G)$.

The following corollary is obtained by combining these findings with the description of the structure of $\mathcal{M}(X,G)$ provided in Theorem~\ref{thm:1.4}. 
\begin{corollary}\label{cor:1.5}
Let $(X,G)$ be an asymptotically entropy expansive system with the specification property. For $U\subset \mathcal{M}\left( X,G\right)$, denote
$$\mathcal{H}(X,G,U):=\left\lbrace h_{\nu}(X, G): \nu\in U\cap\mathcal{M}_{e}\left( X,G\right)\right\rbrace. $$ 
Then, for every $\mu\in\mathcal M(X,G)$ and every neighborhood
$U$ of $\mu$, we have
\begin{align*}
\begin{cases}
\mathcal{H}(X,G,U) \supset \big[0, h_\mu(X, G)\big], & \text{if } h_\mu(X, G)< h(X, G); \\
\mathcal{H}(X,G,U) \supset \big[0, h_\mu(X, G)\big), & \text{if } h_\mu(X, G) = h(X, G).
\end{cases}
\end{align*}
\end{corollary}
\section{Preliminaries}\label{sec:Pre}
In this section, we present some notions and properties concerning amenable groups.

Let $\mathcal{F}(G)$ be the collection of finite subsets of $G$. A sequence $\{F_n\}_{n=1}^\infty \subset \mathcal{F}(G)$ is said to be a F{\o}lner sequence if for any $e \neq g \in G$, we have
$$\lim_{n \to \infty} \frac{|g F_n \triangle F_n|}{|F_n|} = 0,$$
where $gF=\left\lbrace gf:f\in F\right\rbrace $, and $|\cdot|$ denotes the cardinality of a set. We say $G$ is amenable if it admits a F{\o}lner sequence.

Let $K \in \mathcal{F}(G)$. The $K$ boundary of $F$ is defined by
\[
\partial_K F := \{c \in G : Kc \cap F \neq \emptyset, \; Kc \cap (G \setminus F) \neq \emptyset\}.
\]
We say the set $F$ is $(K, \delta)$-invariant if $\frac{|\partial_K F|}{|F|} < \delta$. Equivalently, $\{F_n\}_{n=1}^\infty$ is a F{\o}lner sequence if for every $K\in\mathcal{F}(G)$, $\delta>0$, there exists $N>0$ such that for every $n>N$, $F_n$ is $(K,\delta)$-invariant.

A F{\o}lner sequence $\{F_n\}_{n=1}^\infty$ is said to be \textit{tempered} if there exists a constant $L$ independent of $n$ such that $|\bigcup_{k<n}F_k^{-1}F_n|\leq L|F_n|.$
\subsection{Topological entropy and expansiveness}
In this subsection, we recall the definition of topological
entropy for $G$-topological dynamical systems and some related
notions. Throughout the paper, we use the convention $\log=\log_e$.
We fix a tempered F{\o}lner sequence $\{K_n\}_{n\ge1}$ and
compute all scale-dependent entropy quantities with respect to
this sequence.

Let $(X,\rho)$ be a compact metric space. For any
$\varepsilon>0$ and any $F\in\mathcal F(G)$, a set
$E\subset X$ is called $(F,\varepsilon)$-separated if, for any
distinct $x,y\in E$, there exists $g\in F$ such that
$\rho(gx,gy)>\varepsilon$.
For any subset $Y\subset X$, let $s(Y,F,\varepsilon)$ denote the
maximal cardinality of an $(F,\varepsilon)$-separated subset of
$Y$. Define
\[
h(Y,G,\varepsilon)
:=
\limsup_{n\to\infty}
\frac{1}{|K_n|}
\log s(Y,K_n,\varepsilon).
\]
For a compact $G$-invariant subset $Y\subset X$, its topological
entropy is defined by
\[
h(Y,G)
:=
\lim_{\varepsilon\to0}h(Y,G,\varepsilon).
\]
Let $A$ be a compact subset of $X$. For $F\in\mathcal F(G)$, define the Bowen metric $\rho_F$ by
$\rho_F(x,y):=\max_{g\in F}\rho(gx,gy)$. For $\varepsilon>0$, $x\in A$, and $n\ge1$, the set
\[
B_{K_n}(x,\varepsilon)
=
\left\{y\in X:\rho_{K_n}(x,y)<\varepsilon\right\}
\]
is called an $(n,\varepsilon)$-ball of $(X,G)$. For $\delta>0$, define
\[
s(K_n,\delta,\varepsilon,A,x)
:=
\max
\left\{
|E|:
E\subset A\cap B_{K_n}(x,\varepsilon)
\text{ is }(K_n,\delta)\text{-separated}
\right\}.
\]
Let
$s(K_n,\delta,\varepsilon,A)
:=
\sup_{x\in A}s(K_n,\delta,\varepsilon,A,x)$,
and define
\[
h(A,\varepsilon)
:=
\lim_{\delta\to0}
\limsup_{n\to\infty}
\frac{1}{|K_n|}
\log s(K_n,\delta,\varepsilon,A).
\]
For $\mu\in\mathcal M(X,G)$, define
\[
h_{\mu,\mathrm{loc}}(\varepsilon,G)
:=
\lim_{\sigma\to1^-}
\inf
\left\{
h(A,\varepsilon):
A\subset X \text{ is compact and } \mu(A)>\sigma
\right\},
\]
and
\[
h_{\mathrm{loc}}(\varepsilon,G)
:=
\sup_{\mu\in\mathcal M(X,G)}
h_{\mu,\mathrm{loc}}(\varepsilon,G).
\]
We denote by $h_{\mu,\mathrm{loc}}(\varepsilon,G)$ the $\varepsilon$-local entropy of $\mu$ and by $h_{\text{loc}}(\varepsilon,G)$ the $\varepsilon$-local entropy of $(X,G)$. Both depend on the selected F{\o}lner sequence.

We now recall the definition of asymptotic entropy expansiveness.
\begin{definition}\cite[Definition 3.1]{RS16}
	The system $(X,G)$ is said to be asymptotically entropy expansive
	if there exists a F{\o}lner sequence with respect to which
	\[
	\lim_{\varepsilon\to0}
	h_{\mathrm{loc}}(\varepsilon,G)=0.
	\]
\end{definition}
When $(X,G)$ is asymptotically entropy expansive, the fixed
tempered F{\o}lner sequence $\{K_n\}_{n\ge1}$ is chosen so that
$\lim_{\varepsilon\to0}
h_{\mathrm{loc}}(\varepsilon,G)=0$.
We shall use the following entropy estimate.
\begin{proposition}\cite[Theorem 2]{RS16}
	Let $(X,G)$ be a topological dynamical system. Then, for every
	$\varepsilon>0$,
	\[
	h(X,G)
	\le
	h(X,G,\varepsilon)
	+
	h_{\mathrm{loc}}(\varepsilon,G).
	\]
\end{proposition}
\subsection{Some facts from symbolic dynamics}
In this subsection, we introduce some facts on symbolic dynamics.

Let $\mathcal{A}$ be a finite set equipped with the discrete topology. The group $G$ acts on $\mathcal{A}^G$ in a standard way, known as the shift action, which is defined by $(gx)(h) = x(hg)$ for all $g,h\in G$ and $x\in\mathcal{A}^G$. Endowed with the product topology and this shift action, $\mathcal{A}^G$ becomes a zero-dimensional dynamical system, called the full shift over $\mathcal{A}$. A symbolic dynamical system over $\mathcal{A}$ is defined as any closed and $G$-invariant subset $X$ of the full shift.

In the following we will need the concepts of a block and its associated cylinder set. Take a finite subset $E\subset G$. A block with domain $E$ is simply an element $B\in\mathcal{A}^E$. The cylinder set determined by $B$ is given by
\[[B] = \{x\in\mathcal{A}^G : x|_E = B\}.\]
For a symbolic system $X$, we denote by $X_E = \{B \in \mathcal{A}^E : X \cap [B] \neq \emptyset\}$ the collection of all blocks of domain $E$ that appear somewhere in $X$.

Given a symbolic system $X$, one can build a \textit{topological factor} in the form of another symbolic system $Y$ over a finite alphabet $\Upsilon$ by means of a \textit{sliding block code with finite horizon}. Such a factor is specified by a mapping $\Pi : X_E \to \Upsilon$, called the \textit{code}, where the finite set $E$ is referred to as the \textit{horizon}. The code $\Pi$ induces a map $\pi : X \to \Upsilon^G$ via $x \mapsto y$ with
\[y(g)=\Pi\bigl((gx)|_E\bigr).\]
Then $Y = \pi(X)$ is a closed, shift-invariant subset of $\Upsilon^G$. In practice, to verify that $Y$ is indeed a topological factor of $X$, one checks whether every $y\in Y$ is \textit{determined} term by term from some $x\in X$ by an algorithm with \textit{finite horizon}. That is, we must be able to decide the symbol $y(g)$ based solely on:
\begin{enumerate}
	\item $(gx)|_E$: the symbolic content of $x$ within the copy of
	$E$ shifted by $g$;
	\item a \textit{finite amount of additional information}: namely the mapping $\Pi$, which serves as a finite set of rules telling how to compute $\Pi(B)$ from any admissible block $B\in\mathcal{A}^E$.
\end{enumerate}
We will employ the concept of topological entropy for actions of amenable groups, mainly in the symbolic setting, where it works analogously to the case of $\mathbb{Z}$-actions. In \cite{DHZ19}, the authors use the convention $\log=\log_2$.
\begin{definition}\cite[Definition 2.11]{DHZ19}
	Let $\{F_n\}_{n=1}^\infty$ be a F{\o}lner sequence. The \textit{topological entropy} of a symbolic dynamical system $X$ is defined as the limit
	\[
	{h}(X) = \lim_{n \to \infty} \frac{1}{|F_n|} \log N(F_n),
	\]
	where $N(F_n)$ denotes the $F_n$-\textit{complexity} of $X$, i.e., $N(F_n) = |X_{F_n}|$, the number of distinct blocks of domain $F_n$ that occur in $X$.
\end{definition}
\subsection{Quasitilings and tilings of amenable groups}
In this subsection, we present several facts concerning quasitilings and tilings for amenable groups.

Quasi-tiling theory, developed by Ornstein and Weiss \cite{OW87}, provides a powerful framework for studying amenable group actions.

A finite family \( A_1, A_2, \cdots, A_k \in \mathcal{F}(G) \) is said to be $\delta$-disjoint if there exist subsets \(\{B_1, B_2, \cdots, B_k\} \subset \mathcal{F}(G)\) satisfying:
\begin{itemize}
	\item \( B_i \subset A_i \) for \( i = 1, 2, \cdots, k \),
	\item \( B_i \cap B_j = \emptyset \) for \( 1 \leq i < j \leq k \),
	\item \(\frac{|B_i|}{|A_i|} > 1 - \delta \quad\text{for } i = 1, 2, \cdots, k.\)
\end{itemize}
For \(\alpha \in (0, 1]\), the collection \(\{A_1, A_2, \cdots, A_k\}\) is said to $\alpha$-cover a set \(A \in \mathcal{F}(G)\) if
\[
\frac{\left|A \cap \left(\bigcup_{i=1}^k A_i\right)\right|}{|A|} \geq \alpha.
\]
A collection \(\{A_1, A_2, \cdots, A_k\} \subset \mathcal{F}(G)\) is called a $\delta$-quasi-tiling of \(F \in \mathcal{F}(G)\) if there exist sets \(\{C_1, C_2, \cdots, C_k\} \subset\mathcal{F}(G)\) such that:
\begin{itemize}
	\item \( A_i C_i \subset F \) and the family \(\{A_i c \mid c \in C_i\}\) is $\delta$-disjoint for each \(i = 1, 2, \cdots, k\),
	\item \( A_i C_i \cap A_j C_j = \emptyset \) for \( 1 \leq i \neq j \leq k \),
	\item The collection \(\{A_i C_i \mid i = 1, 2, \cdots, k\}\) forms a $(1 - \delta)$-cover of \(F\).
\end{itemize}
The sets \(C_1, C_2, \cdots, C_k\) are referred to as the tiling centers. 

The following direct consequence of
\cite[Lemma~9.4.14]{C15}, expressed in terms of
$(K,\delta)$-invariance, will be used below.
\begin{proposition}\label{pro:quasi-tiling}
Let $0<\eta\le\frac12$. Then there exists an integer
$s_0=s_0(\eta)\ge1$ such that, for every integer $s\ge s_0$,
the following holds.
Let $F_1,F_2,\dots,F_s,D\in\mathcal F(G)$ be nonempty finite sets.
Assume that
\[
F_j
\text{ is }
\left(F_i,\eta^{2s}\right)\text{-invariant}
\qquad
\text{for all }1\le i<j\le s,
\]
and that
\[
D
\text{ is }
\left(F_i,\eta^{2s}\right)\text{-invariant}
\qquad
\text{for every }i=1,\dots,s.
\]
Then $D$ can be $\eta$-quasi-tiled by
$F_1,F_2,\dots,F_s$.
\end{proposition}
For later use, we record the following reformulation of the
geometric part of \cite[Lemma~5.2]{RTZ23} in the notation of
the present paper.
\begin{lemma}\label{lem:quasi-tiling-refinement}
Let $t\ge1$ be an integer, let
$0<\gamma<\frac{1}{12}$, and let
$F\in\mathcal F(G)$ satisfy
$e\in F$ and $F=F^{-1}$.
Let $A_1,\dots,A_t,S\in\mathcal F(G)$ be nonempty finite sets.

Assume that, for every $i=1,\dots,t$,
$A_i\text{ is }\left(F,\frac{\gamma}{|F|}\right)\text{-invariant}$,
and that $S$ admits a
$\frac{\gamma}{|F|}$-quasi-tiling by
$A_1,\dots,A_t$ with center sets $C_1,\dots,C_t$.
Then there exist subsets
\[
J_{i,c}\subset A_i
\quad\text{for } i=1,\dots,t
\text{ and } c\in C_i,
\]
such that, if we define
$\mathcal P_S:=
\left\{
J_{i,c}c:
i=1,\dots,t,\ c\in C_i
\right\}
~~
and
~~
\widetilde S
:=
\bigcup_{P\in\mathcal P_S}P$,
then the following properties hold:
\begin{enumerate}
    \item For any distinct $P,P'\in\mathcal P_S$,
    $FP\cap P'=\varnothing$.

    \item The elements of $\mathcal P_S$ are pairwise disjoint,
    and
    $|J_{i,c}|>(1-3\gamma)|A_i|$
    for every $i=1,\dots,t$ and every $c\in C_i$.

    \item
    $
    \widetilde S\subset S,
    ~~
    |\widetilde S|>(1-4\gamma)|S|,
    ~~
    F\widetilde S\subset S$.
\end{enumerate}
\end{lemma}
We also recall the notion of a quasitiling of the whole group
$G$.
\begin{definition}\cite[Definition 3.1]{DHZ19}
	A quasitiling $\mathcal T$ of $G$ is determined by the following
	two objects:
	\begin{enumerate}
		\item a finite collection
		$\mathcal S(\mathcal T)\subset\mathcal F(G)$
		of finite subsets of $G$ containing the identity element $e$,
		called the shapes;
		
		\item a finite family
		$\mathcal C(\mathcal T)
		=
		\{C(S):S\in\mathcal S(\mathcal T)\}$
		of pairwise disjoint subsets of $G$, called the center sets.
\end{enumerate}
	The quasitiling is the family
	$\mathcal T=\{(S,c):S\in\mathcal S(\mathcal T),\ c\in C(S)\}$.
	We require that the map
	$(S,c)\mapsto Sc$
	is injective. Depending on the context, a tile of $\mathcal T$
	means either the set $Sc$ or the pair $(S,c)$.
\end{definition}
 We next recall the notions of a tiling and a congruent sequence
of tilings.
\begin{definition}
\cite[Definitions~3.1 and~3.2]{DHZ19}
A quasitiling $\mathcal T$ is called an \emph{exact tiling},
or simply a \emph{tiling}, if it is a partition of $G$,
equivalently,
$G=\bigcup_{T\in\mathcal T}T$
and
\[
T\cap T'=\emptyset
~~
\text{for all distinct }T,T'\in\mathcal T.
\]
Let $\{\mathcal T_k\}_{k\ge1}$ be a sequence of tilings of $G$.
We say that $\{\mathcal T_k\}_{k\ge1}$ is \emph{congruent} if,
for every $k\ge1$, each tile of $\mathcal T_{k+1}$ is a union
of tiles of $\mathcal T_k$.
\end{definition}
The following definition gives the topological entropy of the system $(X_{\mathcal{T}}, G)$.
\begin{definition}\cite[Definition 3.5]{DHZ19}\label{def:2.2}
	Let $\mathcal{T}$ be a quasitiling of $G$ with shape set $\mathcal S(\mathcal{T})$ and tiling centers $C(S)$ for each $S\in \mathcal S(\mathcal{T})$. 
	We encode $\mathcal{T}$ as a configuration $x_{\mathcal{T}} \in \mathcal{A}^G$, where the alphabet is $\mathcal{A}= \mathcal S(\mathcal{T}) \cup \{0\}$, by setting
	\[
	x_{\mathcal{T}}(g) = 
	\begin{cases}
	S, & \text{if } g \in C(S) \text{ for some } S\in \mathcal S(\mathcal{T}),\\[4pt]
	0, & \text{otherwise}.
	\end{cases}
	\]
	The shift action of $G$ on $\mathcal{A}^G$ is given by $(g y)(h) = y(h g)$ for all $y\in\mathcal{A}^G$ and $g,h\in G$. 
	Consider the orbit of $x_{\mathcal{T}}$ under this action: $\{ g x_{\mathcal{T}} \mid g \in G \}$. 
	Take its closure in $\mathcal{A}^G$ with respect to the product topology; denote this closed, shift-invariant set by $X_{\mathcal{T}}$:
	$$X_{\mathcal{T}} = \overline{ \{\, g x_{\mathcal{T}} : g \in G \,\} }.$$
	The dynamical system $(X_{\mathcal{T}}, G)$ is called the \emph{dynamical quasitiling} generated by $\mathcal{T}$. 
	If $\mathcal{T}$ is a tiling, then $X_{\mathcal{T}}$ is called a \emph{dynamical tiling}. Following the convention used throughout this paper, the \emph{entropy} of $\mathcal{T}$, denoted $h(\mathcal{T})$, is defined as the topological entropy of the system $(X_{\mathcal{T}}, G)$.
\end{definition}
Using the boundary-invariance convention of
\cite[Lemma~2.4]{RTZ23}, Theorem~5.2 of \cite{DHZ19}
gives the following zero-entropy tiling lemma.
\begin{lemma}\label{lem:2.6}
	Let $\{\varepsilon_k\}_{k\geq 1}$ be a sequence of positive numbers with 
	$\varepsilon_k \to 0$, and let $\{K_k\}_{k\geq 1}$ be a sequence of finite subsets of $G$. 
	Then there exists a congruent sequence of tilings $\{\mathcal{T}_k\}_{k\geq 1}$ of $G$ 
	such that every shape appearing in $\mathcal{T}_k$ is $(K_k,\varepsilon_k)$-invariant and $h(\mathcal{T}_k)=0$ for each $k$.
\end{lemma}
Although the authors of \cite{DHZ19} use the convention $\log=\log_2$, changing the base of the logarithm only multiplies entropy by a positive constant. In particular, the property of having zero entropy is independent of the base of the logarithm. Hence Lemma~\ref{lem:2.6} applies under our convention $\log=\log_e$.
\begin{lemma}\cite[Lemma 7.1]{RTZ23}\label{lem:2.7}
	Let $K\in \mathcal{F}(G)$, $\theta>0$, and let 
	$\mathcal{T}$ be a tiling of $G$ whose set of shapes is 
	$\mathcal{S}=\{S_1,\ldots,S_\ell\}.$
	Assume that $K$ is 
	$
	\left(\bigcup_{S\in\mathcal{S}} S,\,
	\frac{\theta}{\left|\bigcup_{S\in\mathcal{S}} S\right|}\right)
	\text{-invariant}.
	$
	Define	$\tilde{K}=\bigcup \{\,T\in\mathcal{T} : T\subseteq K\,\}.$
	Then $|\tilde{K}|>(1-\theta)|K|$.
\end{lemma}
\subsection{Invariant measures and metric entropy}
In this subsection, we present the notions of metric entropy and the variational principle in the setting of $G$-topological dynamical systems.

 As $X$ is compact, both $\mathcal{M}\left( X\right)$ and $\mathcal{M}\left( X,G\right)$ are compact metrizable spaces under the weak-* topology \cite[Theorem 6.5 and Theorem 6.10]{W82}.
Denote by $C(X,\mathbb{R})$ the space of real-valued continuous functions on $X$ endowed with the supremum norm $\|\cdot\|_\infty$. Let $f\in C(X,\mathbb R)$ and $\mu\in \mathcal{M}(X)$. We denote $ \langle f, \mu \rangle = \int_X f d\mu$. There exists a countable and separating set of continuous functions $\left\lbrace f_1, f_2, \ldots\right\rbrace $ with $0 \leq f_k \leq 1$ such that
$$D(\mu,\nu):= \| \mu - \nu \|:=\sum_{i=1}^\infty\frac{|\langle f_i, \mu-\nu\rangle|}{2^i}$$
defines a compatible metric for the weak-* topology on $\mathcal{M}\left( X\right)$ \cite[Theorem 6.4]{W82}. By the definition, $D(\mu, \nu) \leq 1 $ for any $\mu,\nu \in \mathcal{M}\left( X\right)$. Since $\mathcal M(X)$ is compact, we denote its diameter with
respect to $D$ by
$D^*:=\max\left\{D(\mu,\nu):
\mu,\nu\in\mathcal M(X)
\right\}$.
\begin{proposition}\cite[Theorem 6.4]{W82}\label{Pro:2.4}
The metric $D$ defined above induces the weak-$*$ topology on
$\mathcal M(X)$ and satisfies
\[
D\left(
\sum_{k=1}^{n}a_k\mu_k,
\sum_{k=1}^{n}a_k\nu_k
\right)
\le
\sum_{k=1}^{n}a_kD(\mu_k,\nu_k).
\]
Here $n\in\mathbb N$,
$\mu_k,\nu_k\in\mathcal M(X)$, and $a_k>0$ for
$k=1,\dots,n$, with
$\sum_{k=1}^{n}a_k=1$.
\end{proposition}
Denote by $\operatorname{ext}(K)$ the set of extreme points of
a convex set $K$. It is well known that
$\mathcal M_e(X,G)=\operatorname{ext}\bigl(\mathcal M(X,G)\bigr)$,
and $\mathcal M(X,G)$ is a Choquet simplex, that is, every
$\mu\in\mathcal M(X,G)$ is the barycenter of a unique probability
measure supported on
$\operatorname{ext}\bigl(\mathcal M(X,G)\bigr)$.
Moreover, $\mathcal{M}_{e}\left( X,G\right)$ is a $G_{\delta}$ subset of $\mathcal{M}(X,G)$. If $\mathcal{M}_{e}\left( X,G\right)$ is dense in $\mathcal{M}\left( X,G\right)$, then $\mathcal{M}_{e}\left( X,G\right)$ is a residual subset of $\mathcal{M}\left( X,G\right)$ and in this case $\mathcal{M}\left( X,G\right)$ is a Poulsen simplex if and only if it is not a singleton. The structure of the Poulsen simplex has been studied in \cite{LOS78}. Some important facts are listed below. Readers are referred to \cite{P01} for more details on Choquet simplices.
\begin{definition}\cite[Section 3.2]{RS16}
Let $\mu\in \mathcal M(X,G)$ and let $\alpha$ be a finite measurable partition of $X$.
For $F\in\mathcal F(G)$, set
$\alpha^F=\bigvee_{g\in F}g^{-1}\alpha$
and
$H_\mu(\alpha)=-\sum_{A\in\alpha}\mu(A)\log\mu(A)$.
The metric entropy of $\mu$ with respect to $\alpha$ is defined by
\[
h_\mu(\alpha,G)
=
\lim_{n\to\infty}\frac{1}{|K_n|}H_\mu(\alpha^{K_n}).
\]
The above limit exists and is independent of the choice of the F{\o}lner sequence.
The metric entropy of $(X,G)$ with respect to $\mu$ is defined by
\[
h_\mu(X,G)=\sup_{\alpha\in \mathcal P_X} h_\mu(\alpha,G),
\]
where $\mathcal P_X$ denotes the set of all finite measurable partitions of $X$.
\end{definition}
Throughout this article, by entropy map we mean the map $\mu\mapsto h_{\mu}(X, G)$ defined on $\mathcal{M}\left( X,G\right).$
\begin{proposition}\cite{O85,OP82}
For any $\mu,\nu\in\mathcal{M}\left( X,G\right)$ and $\lambda\in[0,1]$, we have
$$h_{\lambda\mu+(1-\lambda)\nu}(X, G)=\lambda h_{\mu}(X, G)+(1-\lambda)h_{\nu}(X, G).$$
\end{proposition}
\begin{proposition}[Variational Principle]\cite{O85,OP82}
For any system $(X,G)$, we have
$$h(X, G)=\sup\left\lbrace h_{\mu}(X, G):\mu\in\mathcal{M}\left( X,G\right)\right\rbrace =\sup\left\lbrace h_{\mu}(X, G):\mu\in\mathcal{M}_{e}\left( X,G\right)\right\rbrace.$$		
\end{proposition}
\begin{proposition}\cite[Theorem 1]{RS16}\label{pro:2.11}
If $(X,G)$ is asymptotically entropy expansive, then the entropy map $\mu\mapsto h_{\mu}(X, G)$ is upper semi-continuous in $\mathcal{M}\left( X,G\right)$. As a corollary, there is $\mu_{M}\in\mathcal{M}_{e}\left( X,G\right)$, which is called a measure of maximal entropy, such that $h_{\mu_{M}}(X, G)=h(X, G)$.
\end{proposition}
\subsection{Specification property}
In this subsection, we recall the specification property for general group actions, as presented in \cite[Section 6]{CL15}.
\begin{definition}
	Let $(X,G)$ be a topological dynamical system, and let $\rho$ be a metric on $X$. 
	The system $(X,G)$ is said to have the specification property if, for every 
	$\varepsilon>0$, there exists a nonempty finite subset $K\subset G$ such that for any 
	finite subsets $F_1,F_2,\ldots,F_m\subset G$ with
	\begin{center}
		$K F_i\cap F_j=\varnothing,\quad 1\leq i\neq j\leq m,$
	\end{center}
	and for any points $x_1,x_2,\ldots,x_m\in X$, one can find $y\in X$ satisfying
	\begin{center}
		$\rho(sx_i,sy)\leq \varepsilon,\quad s\in F_i,\ 1\leq i\leq m.$
	\end{center}
\end{definition}
Ren, Tian and Zhou proved in \cite{RTZ23} that the specification property implies entropy denseness. We now recall the definition of entropy-dense.
\begin{definition}\cite[Definition 2.7]{PS05}
Let $\mu\in\mathcal M(X,G)$. We say that $\mu$ is
entropy-approachable by ergodic measures if, for every
$\eta>0$ and every $h<h_\mu(X,G)$, there exists
$\nu\in\mathcal M_e(X,G)$ such that
\[
D(\mu,\nu)<\eta
\quad\text{and}\quad
h_\nu(X,G)>h.
\]
The system $(X,G)$ is entropy-dense if every $\mu\in\mathcal{M}\left( X,G\right)$ is entropy-approachable by ergodic measures.
\end{definition}
\begin{proposition}\cite[Theorem 2.8]{RTZ23}\label{pro:2.15}
If $(X,G)$ has the specification property, then $(X,G)$ is entropy-dense.
\end{proposition}
\subsection{Empirical measures}
In this subsection, we present several basic facts concerning empirical measures.
For $x\in X$ and a nonempty finite set $F\subset G$, define
\[
\mathcal E_F(x)
:=
\frac{1}{|F|}
\sum_{g\in F}\delta_{gx},
\]
where $\delta_{x}$ is the Dirac mass at $x$.
In particular, for a F{\o}lner sequence
$\{F_n\}_{n\ge1}$, we write
$\mathcal E_{F_n}(x)
=
\frac{1}{|F_n|}
\sum_{g\in F_n}\delta_{gx}$.
For a set $U\subset\mathcal M(X)$ and a nonempty finite set
$F\subset G$, define
\[
X_{F,U}
:=
\left\{
x\in X:\mathcal E_F(x)\in U
\right\}.\]
Let $\mu\in\mathcal{M}\left( X,G\right)$ and $\eta>0$. Denote
$$B(\mu, \eta)=\left\lbrace \nu\in\mathcal{M}\left( X,G\right):D(\mu,\nu)<\eta \right\rbrace.$$
\begin{definition}\cite[Definition 2.3]{PS05}
	Let $\mu\in \mathcal{M}(X)$. An $f$-neighborhood of $\mu$ is a set of the form
	$$\mathcal{F}^{(\alpha)}
	:=\mathcal{F}^{(\alpha)}(k)=
	\left\{
	\nu\in \mathcal{M}(X):
	\left|\langle f_i,\mu\rangle-\langle f_i,\nu\rangle\right|
	\le \alpha~\varepsilon_i,~~ i=1,\dots,k
	\right\},$$
	where $
	\alpha>0, \varepsilon_i>0, f_i\in C(X,\mathbb R), \|f_i\|_\infty\le 1$
	for each $i=1,\dots,k$. Here
	$
	\langle f_i,\nu\rangle
	:=
	\int_X f_i\,d\nu
	$
	and $\|f_i\|_\infty	:=\sup_{x\in X}|f_i(x)|.$
	
The weak$^*$ topology on $\mathcal{M}(X)$ admits the family of all $f$-neighborhoods as a neighborhood basis, and this topology is adopted throughout the paper.
\end{definition}
For a nonempty finite set $F\subset G$ and $\delta,\gamma>0$,
a set $E\subset X$ is called $(F,\delta,\gamma)$-separated if,
for every distinct $x,y\in E$,
$\left|\left\{g\in F:\rho(gx,gy)>\gamma\right\}\right|\geq \delta |F|$.
The following proposition, proved in
\cite[Proposition~3.7]{RTZ23}, will be used in our entropy
estimates for amenable group actions.
\begin{proposition}\label{pro:2.17}
	Let $(X,G)$ be a topological dynamical system, and let
	$\{K_n\}_{n=1}^{\infty}$ be a F{\o}lner sequence.
	Suppose that
	$\mu\in\mathcal{M}_e(X,G)
	~~\text{and}~~h<h_\mu(X,G)$.
	Then there exist $\delta>0$ and $\gamma>0$ such that, for every
	weak-$*$ neighborhood $C$ of $\mu$ in $\mathcal{M}(X)$, there exists
	$n_C^*\in\mathbb{N}$ such that, for every $n\geq n_C^*$, there is a
	finite set
	$I_n\subset X_{K_n,C}$
	which is $(K_n,\delta,\gamma)$-separated and satisfies
	$|I_n|\geq e^{|K_n|h}$.
\end{proposition}
\section{Almost entropy-approximability}
In this section, we prove Theorem~\ref{thm:1.3}.
One could prove the theorem directly for an arbitrary invariant
measure by approximating it with finite convex combinations of
ergodic measures. Instead, we use Proposition~\ref{pro:2.15}
to simplify the exposition. It therefore suffices to establish
the desired approximation property for ergodic measures.

The proof of Proposition~\ref{pro:3.1} adapts the construction
used in the proof of \cite[Proposition~7.2]{RTZ23}. In addition, we introduce a coding argument to establish the
entropy upper bound in part~\textup{(2)}.
\begin{proposition}\label{pro:3.1}
Let $G$ be an infinite countable amenable group, and let
$\{K_m\}_{m=1}^\infty$ denote the fixed tempered F{\o}lner
sequence chosen above. Let \((X, G)\) be a dynamical system with the specification property. Suppose \(\mu_0 \in \mathcal{M}_e(X, G)\), \(h_0 \in (0, h_{\mu_0}(X,G))\), and \(\eta_0, \beta_0, \varepsilon_0 > 0\). Then there exist \(\gamma \in (0, \varepsilon_0)\) and a compact \(G\)-invariant subset \(\Lambda = \Lambda(\mu_0, h_0, \eta_0, \beta_0, \gamma)\) such that:
\begin{enumerate}
	\item\label{pro 3.1(1)} \(\mathcal{M}(\Lambda, G)\subset B(\mu_0,\eta_0)\);
	\item \label{pro 3.1(2)}\(h(\Lambda, G) > h_0\) and \(h(\Lambda, G, \gamma) < h_0 + \beta_0\).
\end{enumerate}
\end{proposition}
When $h_{\mu_0}(X,G)=0$, we have the following:
\begin{proposition}\label{pro:3.1-zero}
	Let $G$ be an infinite countable amenable group, and let
	$\{K_m\}_{m=1}^\infty$ denote the fixed tempered F{\o}lner
	sequence chosen above. 
	Let $(X,G)$ be a dynamical system with the specification property.
	Suppose that
	$\mu_0\in\mathcal M_e(X,G),~~h_{\mu_0}(X,G)=0,
	$
	and let $\eta_0>0$. Then there exists a nonempty compact $G$-invariant
	subset $\Lambda\subset X$ such that
	$\mathcal M(\Lambda,G)\subset B(\mu_0,\eta_0)$.
\end{proposition}
We shall prove Proposition~\ref{pro:3.1} in the following subsections.
The proof of Proposition~\ref{pro:3.1-zero} is analogous to that of
Proposition~\ref{pro:3.1} and is therefore omitted (the modification
required for the zero-entropy case is described after \eqref{3.4}, and
no entropy estimate is needed).
\subsection{Construction}\label{con:3.1}
In this subsection, we construct the auxiliary set $Z^\#$. 

Choose an $f$-neighborhood $\mathcal F^{(1)}$ of $\mu_0$ such that
$\mathcal F^{(1)}\cap \mathcal M(X,G)\subset B(\mu_0,\eta_0)$.
Fix $\{f_j,\varepsilon_j:1\le j\le p\}$ for $\mathcal F^{(1)}$ and denote
$\varepsilon_{\min}:=\min\{\varepsilon_j:1\le j\le p\}$.

Let
$$
0<\beta<
\frac{1}{20}
\min\{\beta_0,\ h_{\mu_0}(X,G)-h_0,\ h_0\},
\qquad
h_1:=h_0+10\beta.
$$
Then
\begin{align}\label{3.1}
h_1+\beta<h_{\mu_0}(X,G), \qquad h_1+3\beta=h_0+13\beta<h_0+\beta_0.
\end{align}
Choose $h^*$ such that
$h_1+\beta<h^*<h_{\mu_0}(X,G)$ and choose $N_0\in\mathbb N$ such that
\[
e^{(h_1+\beta)N}>e^{h_1N}+1
\qquad
\text{for every }N\ge N_0.
\]
By Proposition~\ref{pro:2.17}, applied to $h^*$ for $\mu_0$
and the neighborhood $\mathcal F^{(1/20)}$, there exist constants
$\delta^*>0,~~\varepsilon^*>0$ and $n^*_{\mathcal F^{(1/20)}}\in\mathbb N$ such that, for every
$n\ge n^*_{\mathcal F^{(1/20)}}$, there exists a finite set
$I_n\subset X_{K_n,\mathcal F^{(1/20)}}$
which is $(K_n,\delta^*,\varepsilon^*)$-separated and satisfies
$|I_n|\ge e^{h^*|K_n|}$.
Set $n^*:=n^*_{\mathcal F^{(1/20)}}$. 

Let $0<\alpha<\frac{\varepsilon_0}{3}$. Choose $\Delta>0$ such that
\[
2\Delta<\alpha,\qquad
3\Delta<\varepsilon^*,
\]
and
\[
\rho(x,y)\le\Delta
\quad\Longrightarrow\quad
|f_j(x)-f_j(y)|<\frac{\varepsilon_j}{5},
\qquad j=1,\dots,p.
\]
Let $F_\Delta\subset G$ be the finite set given by the
specification property corresponding to the tracing accuracy $\Delta$.
After replacing $F_\Delta$ by
$F_\Delta\cup F_\Delta^{-1}\cup\{e\}$, we may assume that
$e\in F_\Delta
~~\text{and}~~
F_\Delta=F_\Delta^{-1}$.

Let $Q=\{q_1,\dots,q_{r_Q}\}\subset X$ be a finite
$\frac{\alpha}{4}$-spanning set with $r_Q\ge2$.
Let $\{\alpha_k\}_{k\ge1}$ be a sequence of positive numbers strictly decreasing to $0$ such that
\begin{align}\label{3.2}
\alpha_1
<
\min\left\{
\frac{10\beta}{3h_1},\
\frac{\delta^*}{6},\
\frac{1}{12},\
\frac{\varepsilon_{\min}}{50},\
\frac{\beta}{8\log r_Q}
\right\}.
\end{align}
For each $k\ge1$, set
$\eta_k:=\frac{\alpha_k}{|F_\Delta|}$.
Since $F_\Delta\neq\emptyset$ and
$\alpha_k\le\alpha_1<\frac{1}{12}$, we have
$0<\eta_k<\frac12$.
Let
$t_k:=s_0(\eta_k)$,
where $s_0(\eta_k)$ is provided by
Proposition~\ref{pro:quasi-tiling}.

Using the F{\o}lner property and the fact that
$|K_n|\to\infty$, we recursively choose integers
\[
n_k^*<n_{k,1}<n_{k,2}<\cdots<n_{k,t_k}<n_{k+1}^*
\qquad (k\ge1)
\]
with the following properties:
$n_k^*\ge n^*$,
and, for every $n\ge n_k^*$,
$K_n
\text{ is }
\left(F_\Delta,\eta_k\right)\text{-invariant}$
and
$(1-3\alpha_k)|K_n|\ge N_0$.
Moreover,
\[
K_{n_{k,j}}
\text{ is }
\left(
K_{n_{k,i}},\eta_k^{2t_k}
\right)\text{-invariant}
\qquad
\text{for all }1\le i<j\le t_k.
\]
Set
\[
L_k:=\bigcup_{j=1}^{t_k}K_{n_{k,j}}
\qquad\text{and}\qquad
\xi_k:=\eta_k^{2t_k}.
\]
Since $t_k\ge1$ and $\eta_k\to0$, we have
$0<\xi_k\le\eta_k^2\rightarrow 0$.
Applying Lemma~\ref{lem:2.6} to the sequences
$\{L_k\}_{k\ge1}$ and $\{\xi_k\}_{k\ge1}$, we obtain a
congruent sequence of tilings
$\{\mathcal T_k\}_{k\ge1}$ of $G$ such that
$h(\mathcal T_k)=0~~\text{for every }k\ge1$.
Denote by
$
\mathcal S_k:=\mathcal S(\mathcal T_k)
$
the finite shape family of $\mathcal T_k$. Then every
$S\in\mathcal S_k$ is $(L_k,\xi_k)$-invariant.

For every $j=1,\dots,t_k$, since
$K_{n_{k,j}}\subset L_k$,
we have
$\partial_{K_{n_{k,j}}}S
\subseteq
\partial_{L_k}S$.
Hence every $S\in\mathcal S_k$ is
$\left(
K_{n_{k,j}},\eta_k^{2t_k}
\right)\text{-invariant}
~~
\text{for every }j=1,\dots,t_k.
$
Together with the choice of
$K_{n_{k,1}},\dots,K_{n_{k,t_k}}$, all the assumptions of
Proposition~\ref{pro:quasi-tiling} are satisfied. Therefore,
every $S\in\mathcal S_k$ can be $\eta_k$-quasi-tiled by
$K_{n_{k,1}},K_{n_{k,2}},\dots,K_{n_{k,t_k}}$.
Take $\theta>0$ such that
\begin{align}\label{3.3}
\theta<
\min\left\{
\frac{\varepsilon_{\min}}{20},\
\frac{\beta}{2\log r_Q},\
1-\frac{h_0}{h_1}
\right\}.
\end{align}
Choose $k^\#$ sufficiently large such that, for every
$k\ge k^\#$,
\[
8\alpha_k<\frac{\varepsilon_{\min}}{5},\qquad
(\theta+4\alpha_k)\log r_Q<\beta,\qquad
h_1(1-\theta)(1-4\alpha_k)>h_0.
\]

From now on, fix $k\ge k^\#$. For every
$S\in\mathcal S_k$, fix an $\eta_k$-quasi-tiling of $S$ by
$K_{n_{k,1}},K_{n_{k,2}},\dots,K_{n_{k,t_k}}$,
and denote the corresponding center sets by
$C_{k,S,1},C_{k,S,2},\dots,C_{k,S,t_k}$.

Since
$\eta_k=\frac{\alpha_k}{|F_\Delta|}$
and, for every $i=1,\dots,t_k$,
$K_{n_{k,i}}
\text{ is }
\left(
F_\Delta,
\frac{\alpha_k}{|F_\Delta|}
\right)\text{-invariant}$,
Lemma~\ref{lem:quasi-tiling-refinement} applies. Hence, for every
$i=1,\dots,t_k$ and every $c\in C_{k,S,i}$, there exists a subset
$J_{k,S,i,c}\subset K_{n_{k,i}}$.

Set
\[
\mathcal P_{k,S}
:=
\left\{
J_{k,S,i,c}c:
c\in C_{k,S,i},\ i=1,\dots,t_k
\right\}
\quad\text{and}
\quad
\widetilde S_{k,S}
:=
\bigcup_{P\in\mathcal P_{k,S}}P.
\]
Then the following properties hold:
\begin{enumerate}
    \item If $P,P'\in\mathcal P_{k,S}$ and $P\neq P'$, then
    $F_\Delta P\cap P'=\varnothing$.

    \item The elements of $\mathcal P_{k,S}$ are pairwise
    disjoint, and
    $
    |J_{k,S,i,c}|
    >
    (1-3\alpha_k)|K_{n_{k,i}}|
    $
    for every $i=1,\dots,t_k$ and every
    $c\in C_{k,S,i}$.

    \item The refined tiles satisfy
    $
    \widetilde S_{k,S}\subset S,
    ~~
    |\widetilde S_{k,S}|
    >
    (1-4\alpha_k)|S|
    ~~\text{and}~~
    F_\Delta\widetilde S_{k,S}\subset S.$
\end{enumerate}
For each $i=1,\dots,t_k$, since
$n_{k,i}>n_k^*\ge n^*$,
choose once and for all, by
Proposition~\ref{pro:2.17}, a finite set
$I_{k,i}\subset
X_{K_{n_{k,i}},\mathcal F^{(1/20)}}$
which is
$\left(
K_{n_{k,i}},
\delta^*,
\varepsilon^*
\right)\text{-separated}$
and satisfies
$|I_{k,i}|
\ge
e^{h^*|K_{n_{k,i}}|}$.
\begin{claim}
For every $S\in\mathcal S_k$, every $i=1,\dots,t_k$, and
every $c\in C_{k,S,i}$, the set $I_{k,i}$ is contained in
$X_{J_{k,S,i,c},\mathcal F^{(1/5)}}$,
is
$\left(
J_{k,S,i,c},
\frac{\delta^*}{2},
\varepsilon^*
\right)\text{-separated},$
and satisfies
$|I_{k,i}|
\ge
e^{h^*|J_{k,S,i,c}|}$.
\end{claim}
\begin{proof}
Put
\[
K:=K_{n_{k,i}},\qquad
J:=J_{k,S,i,c},\qquad
R:=K\setminus J.
\]
Then
$\frac{|R|}{|K|}<3\alpha_k$.
For every $x\in I_{k,i}$ and every $j=1,\dots,p$, if
$R\neq\emptyset$, then
\[
\mathcal E_K(x)
=
\frac{|J|}{|K|}\mathcal E_J(x)
+
\frac{|R|}{|K|}\mathcal E_R(x).
\]
Since $\|f_j\|_\infty\le1$, it follows that
\[
\begin{aligned}
\left|
\left\langle f_j,\mathcal E_J(x)\right\rangle
-
\left\langle f_j,\mathcal E_K(x)\right\rangle
\right|
&=
\frac{|R|}{|K|}
\left|
\left\langle f_j,\mathcal E_J(x)\right\rangle
-
\left\langle f_j,\mathcal E_R(x)\right\rangle
\right|  \\
&\le
2\frac{|R|}{|K|}
<
6\alpha_k.
\end{aligned}
\]
When $R=\emptyset$, the desired estimate is trivial since $J=K$. Hence, by the choice of $\alpha_1$ in \eqref{3.2} and the definition
of $I_{k,i}$, we obtain
\[
\begin{aligned}
\left|
\left\langle f_j,\mathcal E_J(x)\right\rangle
-
\left\langle f_j,\mu_0\right\rangle
\right|
&\le
\left|
\left\langle f_j,\mathcal E_J(x)\right\rangle
-
\left\langle f_j,\mathcal E_K(x)\right\rangle
\right|  \\
&\quad+
\left|
\left\langle f_j,\mathcal E_K(x)\right\rangle
-
\left\langle f_j,\mu_0\right\rangle
\right| \\
&<
6\alpha_k+\frac{\varepsilon_j}{20} \\
&<
\left(\frac{6}{50}+\frac{1}{20}\right)\varepsilon_j
<
\frac{\varepsilon_j}{5}.
\end{aligned}
\]
Therefore
$I_{k,i}\subset
X_{J_{k,S,i,c},\mathcal F^{(1/5)}}$.

Next, let $x,y\in I_{k,i}$ be distinct. Since $I_{k,i}$ is
$(K,\delta^*,\varepsilon^*)$-separated and 
$\alpha_k<\frac{\delta^*}{6}$ by \eqref{3.2}, we have
\[
\begin{aligned}
\left|
\left\{
g\in J:\rho(gx,gy)>\varepsilon^*
\right\}
\right|
&\ge
\delta^*|K|-|K\setminus J| \\
&>
(\delta^*-3\alpha_k)|K| \\
&>
\frac{\delta^*}{2}|K| \ge
\frac{\delta^*}{2}|J|.
\end{aligned}
\]
Thus $I_{k,i}$ is
$\left(
J_{k,S,i,c},
\frac{\delta^*}{2},
\varepsilon^*
\right)\text{-separated}$ and
$|I_{k,i}|\ge e^{h^*|K|}\ge e^{h^*|J|}$.
\end{proof}
Therefore, setting
$\Gamma_{k,S,i,c}^*:=I_{k,i}$,
we obtain a
$\left(
J_{k,S,i,c},
\frac{\delta^*}{2},
\varepsilon^*
\right)\text{-separated}$
set satisfying
$\Gamma_{k,S,i,c}^*
\subset
X_{J_{k,S,i,c},\mathcal F^{(1/5)}}$
and
$
|\Gamma_{k,S,i,c}^*|
\ge
e^{h^*|J_{k,S,i,c}|}$.

By property~\textup{(2)}~above and the choice of $n_k^*$, we have
\[
|J_{k,S,i,c}|
>
(1-3\alpha_k)|K_{n_{k,i}}|
\ge N_0.
\]
Hence, by the choice of $N_0$,
\[
\left\lceil e^{h_1|J_{k,S,i,c}|}\right\rceil
\le
e^{h_1|J_{k,S,i,c}|}+1
<
e^{(h_1+\beta)|J_{k,S,i,c}|}.
\]
Moreover, since $h^*>h_1+\beta$ and
$|\Gamma_{k,S,i,c}^*|
\ge e^{h^*|J_{k,S,i,c}|}$,
we obtain
\[
e^{h_1|J_{k,S,i,c}|}
\le
\left\lceil e^{h_1|J_{k,S,i,c}|}\right\rceil
<
e^{(h_1+\beta)|J_{k,S,i,c}|}
<
|\Gamma_{k,S,i,c}^*|.
\]
Thus, we may choose
$
\Gamma_{k,S,i,c}\subset\Gamma_{k,S,i,c}^*
$
with
$
|\Gamma_{k,S,i,c}|
=
\left\lceil e^{h_1|J_{k,S,i,c}|}\right\rceil,
$
and hence
\begin{align}\label{3.4}
e^{h_1|J_{k,S,i,c}|}
\le
|\Gamma_{k,S,i,c}|
<
e^{(h_1+\beta)|J_{k,S,i,c}|}.
\end{align}
(When $h_{\mu_0}(X,G)=0$, to prove
Proposition~\ref{pro:3.1-zero}, choose a point $x^*$ that is
generic for $\mu_0$ along the fixed tempered F{\o}lner sequence
$\{K_m\}_{m=1}^{\infty}$. By taking $n^*$ sufficiently large, we have
\[
x^*\in X_{K_n,\mathcal F^{(1/20)}}
\qquad\text{for every } n\ge n^*.
\]
The same empirical-measure estimate as above then gives
$x^*\in X_{J_{k,S,i,c},\mathcal F^{(1/5)}}$
for every $S\in\mathcal S_k$, $i=1,\dots,t_k$, and
$c\in C_{k,S,i}$. We therefore take
$\Gamma_{k,S,i,c}
=\Gamma_{k,S,i,c}^*
:=\{x^*\}$.)

Let
\[
\mathscr I_k
:=
\{(S,d,i,c):Sd\in\mathcal T_k,\ i=1,\dots,t_k,\ c\in C_{k,S,i}\}.
\]
Here $Sd\in\mathcal T_k$ means that $S\in\mathcal S_k$ and
$d$ is a center of a tile of shape $S$.
Define
\[
\Omega_k
:=
\prod_{(S,d,i,c)\in\mathscr I_k}\Gamma_{k,S,i,c}.
\]
For $\omega\in\Omega_k$, write its coordinate as
$\omega_{S,d,i,c}\in\Gamma_{k,S,i,c}$.
Let $Z_k^\#(\omega)$ be the set of all $z\in X$ such that
\begin{align}\label{3.5}
\rho_{J_{k,S,i,c}}
\bigl(cdz,\omega_{S,d,i,c}\bigr)
\le \Delta
\end{align}
for every $(S,d,i,c)\in\mathscr I_k$.
For any two distinct global good subtiles
\[
J_{k,S,i,c}cd
\quad\text{and}\quad
J_{k,S',i',c'}c'd',
\]
we have
$
F_\Delta J_{k,S,i,c}cd
\cap
J_{k,S',i',c'}c'd'
=
\emptyset.
$
Indeed, this follows from property~\textup{(1)} when the two
subtiles belong to the same tile, and from property~\textup{(3)}
together with the disjointness of the tiles otherwise.

Fix $\omega\in\Omega_k$. We prove that
$Z_k^\#(\omega)\neq\emptyset$. Let $\mathscr J\subset\mathscr I_k$ be finite and set
\[
x_{S,d,i,c}
:=
d^{-1}c^{-1}\omega_{S,d,i,c},
\qquad
(S,d,i,c)\in\mathscr J.
\] 
By the separation property established above, the specification
property yields $z_{\mathscr J}\in X$ such that
$
\rho\bigl(gz_{\mathscr J},gx_{S,d,i,c}\bigr)
\le\Delta
$
for every $(S,d,i,c)\in\mathscr J$ and every
$g\in J_{k,S,i,c}cd$. Since, for $g=jcd$,
$gx_{S,d,i,c}=j\omega_{S,d,i,c},$
it follows that
$
\rho_{J_{k,S,i,c}}
\bigl(cdz_{\mathscr J},\omega_{S,d,i,c}\bigr)
\le\Delta
$
for all $(S,d,i,c)\in\mathscr J$. Therefore, every finite
subcollection of the conditions defining $Z_k^\#(\omega)$ has
a common solution.

For each $(S,d,i,c)\in\mathscr I_k$, the corresponding tracing
condition defines a closed subset of $X$. By the preceding
argument, these closed sets have the finite intersection
property. Hence, by compactness of $X$,
$Z_k^\#(\omega)\neq\emptyset.$

Endow each finite set $\Gamma_{k,S,i,c}$ with the discrete
topology and $\Omega_k$ with the product topology. Since every
$\Gamma_{k,S,i,c}$ is nonempty and compact, $\Omega_k$ is nonempty and compact by
Tychonoff's theorem.

Define
\[
\mathcal R
:=
\left\{
(z,\omega)\in X\times\Omega_k:
\rho_{J_{k,S,i,c}}
\bigl(cdz,\omega_{S,d,i,c}\bigr)
\le\Delta
\text{ for all }(S,d,i,c)\in\mathscr I_k
\right\}.
\]
For each $(S,d,i,c)\in\mathscr I_k$, the map
$(z,\omega)
\mapsto
\rho_{J_{k,S,i,c}}
\bigl(cdz,\omega_{S,d,i,c}\bigr)$
is continuous. Hence $\mathcal R$ is closed in
$X\times\Omega_k$ and therefore compact. Moreover, since every
$Z_k^\#(\omega)$ is nonempty, $\mathcal R$ is nonempty.

Define
\[
Z^\#:=\pi_X(\mathcal R).
\]
Thus $Z^\#$ is nonempty and compact. For each $z\in Z^\#$, fix one admissible label assignment
$\omega(z)\in\Omega_k$
such that
$(z,\omega(z))\in\mathcal R$.
For $(S,d,i,c)\in\mathscr I_k$, define
$$
a_{k,S,d,i,c}(z):=\omega(z)_{S,d,i,c}.
$$
Then
$$a_{k,S,d,i,c}(z)\in\Gamma_{k,S,i,c}~~\text{and}~~
\rho_{J_{k,S,i,c}}
\bigl(cdz,a_{k,S,d,i,c}(z)\bigr)
\le\Delta.$$
This choice is fixed once and for all for later use. Pick $m_k$ sufficiently large such that, for every $m\ge m_k$, the set $K_m$ is $\left(
\bigcup\mathcal S_k,\,
\frac{\theta}{\left|\bigcup \mathcal S_k\right|}
\right)\text{-invariant}.
$
For every $m\ge m_k$, define
\[
Y_{m,k}
:=\bigcap_{s\in G}
\left\{
x\in X:
\mathcal E_{K_m s}(x)\in\mathcal F^{(4/5)}
\right\}.
\]
Since $\mathcal F^{(4/5)}$ is closed and each map
$x\mapsto\mathcal E_{K_ms}(x)$ is continuous, every set in the
intersection is closed. Hence $Y_{m,k}$ is closed. In addition, $Y_{m,k}$ is $G$-invariant.
\begin{lemma}\label{Lemma:3.3}
For every $m\ge m_k$,
	\[
	Z^\#\subset Y_{m,k}.
	\]
\end{lemma}
\begin{proof}
Fix $m\ge m_k$, $x\in Z^\#$, and $s\in G$. Let
\[
I_{k,m}(s):=
\{Sd\in\mathcal T_k:Sd\subset K_m s\},
\qquad
\widetilde{K_m s}
:=
\bigcup_{Sd\in I_{k,m}(s)}Sd.
\]
Right translation preserves $(L,\delta)$-invariance. Hence
$K_ms$ satisfies the same invariance condition as $K_m$, and
Lemma~\ref{lem:2.7} gives
\begin{equation}\label{3.6}
|\widetilde{K_m s}|>(1-\theta)|K_m|.
\end{equation}
We first estimate the empirical measure on a complete tile $Sd\in I_{k,m}(s)$. 
For such a tile, define the union of its good subtiles by
\[
\mathcal G_{S,d}
:=
\bigcup_{i=1}^{t_k}
\bigcup_{c\in C_{k,S,i}}
J_{k,S,i,c}cd.
\]
Since
$
\mathcal G_{S,d}
=
\widetilde S_{k,S}d,
$
and right translation by $d$ preserves cardinality, property
~\textup{(3)} gives
\[
|\mathcal G_{S,d}|
=
|\widetilde S_{k,S}|
>
(1-4\alpha_k)|S|.
\]
Hence
\begin{equation}\label{3.7}
|Sd\setminus \mathcal G_{S,d}|<4\alpha_k |S|.
\end{equation}
Fix $1\le i\le t_k$ and $c\in C_{k,S,i}$. Since $x\in Z^\#$, the fixed admissible label assignment
associated with $x$ satisfies
\[
a_{k,S,d,i,c}(x)\in\Gamma_{k,S,i,c}
~~
\text{and}
~~
\rho_{J_{k,S,i,c}}
\bigl(cdx,a_{k,S,d,i,c}(x)\bigr)
\le\Delta.
\]
By the choice of $\Delta$, the following estimate holds for
every $j=1,\dots,p$:
\[
\left|
f_j(ucdx)-f_j(u a_{k,S,d,i,c}(x))
\right|
<
\frac{\varepsilon_j}{5},
\qquad
u\in J_{k,S,i,c}.
\]
Averaging over $u\in J_{k,S,i,c}$ gives
\begin{equation}\label{3.8}
\left|
\left\langle f_j,\mathcal E_{J_{k,S,i,c}cd}(x)\right\rangle
-
\left\langle f_j,\mathcal E_{J_{k,S,i,c}}
\bigl(a_{k,S,d,i,c}(x)\bigr)\right\rangle
\right|
<
\frac{\varepsilon_j}{5}.
\end{equation}
Moreover,
\[
a_{k,S,d,i,c}(x)\in
\Gamma_{k,S,i,c}
\subset
X_{J_{k,S,i,c},\mathcal F^{(1/5)}}.
\]
Therefore
\[
\mathcal E_{J_{k,S,i,c}}
\bigl(a_{k,S,d,i,c}(x)\bigr)
\in
\mathcal F^{(1/5)}.
\]
By the definition of $\mathcal F^{(1/5)}$,
\begin{equation}\label{3.9}
\left|
\left\langle f_j,\mathcal E_{J_{k,S,i,c}}
\bigl(a_{k,S,d,i,c}(x)\bigr)\right\rangle
-
\left\langle f_j,\mu_0\right\rangle
\right|
\le
\frac{\varepsilon_j}{5}.
\end{equation}
Combining \eqref{3.8} and \eqref{3.9}, we obtain
\begin{equation}\label{3.10}
\left|
\left\langle f_j,\mathcal E_{J_{k,S,i,c}cd}(x)\right\rangle
-
\left\langle f_j,\mu_0\right\rangle
\right|
<
\frac{2\varepsilon_j}{5}.
\end{equation}
Since the good subtiles inside $Sd$ are pairwise disjoint, the empirical measure on
$\mathcal G_{S,d}$ is a convex combination of the empirical measures on the sets
$J_{k,S,i,c}cd$. Hence \eqref{3.10} implies
\begin{equation}\label{3.11}
\left|
\left\langle f_j,\mathcal E_{\mathcal G_{S,d}}(x)\right\rangle
-
\left\langle f_j,\mu_0\right\rangle
\right|
<
\frac{2\varepsilon_j}{5}.
\end{equation}
Let
\[
\mathcal B_{S,d}:=Sd\setminus \mathcal G_{S,d}.
\]
By convention, any term associated with an empty set in an
empirical-measure decomposition is omitted.
\[
\mathcal E_{Sd}(x)
=
\frac{|\mathcal G_{S,d}|}{|S|}
\mathcal E_{\mathcal G_{S,d}}(x)
+
\frac{|\mathcal B_{S,d}|}{|S|}
\mathcal E_{\mathcal B_{S,d}}(x).
\]
If $\mathcal B_{S,d}\neq\emptyset$, then, since
$\|f_j\|_\infty\le1$, we have
\[
\left|
\left\langle f_j,\mathcal E_{\mathcal B_{S,d}}(x)\right\rangle
-
\left\langle f_j,\mu_0\right\rangle
\right|
\le 2.
\]
If $\mathcal B_{S,d}=\emptyset$, the corresponding term is
omitted. Therefore, in either case, the contribution of
$\mathcal B_{S,d}$ to the above decomposition is bounded by
$\frac{2|\mathcal B_{S,d}|}{|S|}$.
Using \eqref{3.7}, \eqref{3.11}, and
$8\alpha_k<\frac{\varepsilon_{\min}}{5}\le \frac{\varepsilon_j}{5}$,
we get
\[
\begin{aligned}
\left|
\left\langle f_j,\mathcal E_{Sd}(x)\right\rangle
-
\left\langle f_j,\mu_0\right\rangle
\right|
&\le
\frac{|\mathcal G_{S,d}|}{|S|}
\frac{2\varepsilon_j}{5}
+
\frac{2|\mathcal B_{S,d}|}{|S|}  <
\frac{2\varepsilon_j}{5}
+
8\alpha_k <
\frac{3\varepsilon_j}{5}.
\end{aligned}
\]
Therefore, for every complete tile $Sd\in I_{k,m}(s)$ and
every $j=1,\dots,p$,
\begin{equation}\label{3.12}
\left|
\left\langle f_j,\mathcal E_{Sd}(x)\right\rangle
-
\left\langle f_j,\mu_0\right\rangle
\right|
<
\frac{3\varepsilon_j}{5}.
\end{equation}
Since $\widetilde{K_m s}$ is the disjoint union of the complete tiles
$Sd\in I_{k,m}(s)$, the empirical measure
$\mathcal E_{\widetilde{K_m s}}(x)$ is a convex combination of the measures
$\mathcal E_{Sd}(x)$. Hence \eqref{3.12} gives
\begin{equation}\label{3.13}
\left|
\left\langle f_j,\mathcal E_{\widetilde{K_m s}}(x)\right\rangle
-
\left\langle f_j,\mu_0\right\rangle
\right|
<
\frac{3\varepsilon_j}{5}.
\end{equation}
Finally, we compare the averages over $K_m s$ and over $\widetilde{K_m s}$. Put
\[
R_{m,s}:=K_m s\setminus \widetilde{K_m s}.
\]
As above, if $R_{m,s}=\emptyset$, the corresponding term in
the following decomposition is omitted.
By \eqref{3.6},
\[
|R_{m,s}|<\theta |K_m|.
\]
Also,
\[
\mathcal E_{K_m s}(x)
=
\frac{|\widetilde{K_m s}|}{|K_m|}
\mathcal E_{\widetilde{K_m s}}(x)
+
\frac{|R_{m,s}|}{|K_m|}
\mathcal E_{R_{m,s}}(x).
\]
 Since $\|f_j\|_\infty\le1$, we have
\[
\left|
\left\langle f_j,\mathcal E_{K_m s}(x)\right\rangle
-
\left\langle f_j,\mathcal E_{\widetilde{K_m s}}(x)\right\rangle
\right|
\le
2\frac{|R_{m,s}|}{|K_m|}
<
2\theta.
\]
By the choice of $\theta$,
$2\theta<\frac{\varepsilon_j}{5}.$
Combining this with \eqref{3.13}, we obtain
\[
\begin{aligned}
\left|
\left\langle f_j,\mathcal E_{K_m s}(x)\right\rangle
-
\left\langle f_j,\mu_0\right\rangle
\right|
&\le
\left|
\left\langle f_j,\mathcal E_{K_m s}(x)\right\rangle
-
\left\langle f_j,\mathcal E_{\widetilde{K_m s}}(x)\right\rangle
\right|+
\left|
\left\langle f_j,\mathcal E_{\widetilde{K_m s}}(x)\right\rangle
-
\left\langle f_j,\mu_0\right\rangle
\right| \\
&<
\frac{\varepsilon_j}{5}
+
\frac{3\varepsilon_j}{5}
=
\frac{4\varepsilon_j}{5}.
\end{aligned}
\]
This holds for every $j=1,\dots,p$. Therefore
\[
\mathcal E_{K_m s}(x)\in \mathcal F^{(4/5)}.
\]
Since $s\in G$ was arbitrary, $x\in Y_{m,k}$. Thus
\[
Z^\#\subset Y_{m,k}.
\]
\end{proof}
Define
\[
Y:=\bigcap_{m\ge m_k}Y_{m,k}.
\]
Then $Y$ is closed and $G$-invariant, and Lemma \ref{Lemma:3.3} gives $
Z^\#\subset Y$.
\begin{proposition}\label{pro:measure-control}
Let
\[
\Lambda:=\overline{GZ^\#}.
\]
Then $\Lambda$ is a nonempty compact $G$-invariant subset of
$X$, and
$\mathcal M(\Lambda,G)
\subset
B(\mu_0,\eta_0)$.
\end{proposition}
\begin{proof}
Since $Z^\#\subset Y$ and $Y$ is closed and $G$-invariant,
$\Lambda=\overline{GZ^\#}\subset Y$. Since $Z^\#$ is nonempty, it follows immediately that
$\Lambda$ is nonempty, compact, and $G$-invariant.

Let $\nu\in\mathcal M(\Lambda,G)$ and fix $m\ge m_k$.
By the $G$-invariance of $\nu$, for every $f\in C(X)$,
\[
\begin{aligned}
\int_\Lambda
\left\langle f,\mathcal E_{K_m}(y)\right\rangle
\,d\nu(y)
&=
\frac{1}{|K_m|}
\sum_{u\in K_m}
\int_\Lambda f(uy)\,d\nu(y)\\
&=
\frac{1}{|K_m|}
\sum_{u\in K_m}
\langle f,u_*\nu\rangle\\
&=
\langle f,\nu\rangle.
\end{aligned}
\]
For every $y\in\Lambda\subset Y$, taking $s=e$ in the
definition of $Y_{m,k}$ gives
$
\mathcal E_{K_m}(y)\in\mathcal F^{(4/5)}.
$
Hence, for every $j=1,\dots,p$,
\[
\left|
\left\langle f_j,\mathcal E_{K_m}(y)\right\rangle
-
\left\langle f_j,\mu_0\right\rangle
\right|
\le
\frac{4\varepsilon_j}{5}.
\]
Therefore,
\[
\begin{aligned}
\left|
\langle f_j,\nu\rangle
-
\langle f_j,\mu_0\rangle
\right|
&=
\left|
\int_\Lambda
\left(
\left\langle f_j,\mathcal E_{K_m}(y)\right\rangle
-
\left\langle f_j,\mu_0\right\rangle
\right)
\,d\nu(y)
\right|\\
&\le
\int_\Lambda
\left|
\left\langle f_j,\mathcal E_{K_m}(y)\right\rangle
-
\left\langle f_j,\mu_0\right\rangle
\right|
\,d\nu(y)\\
&\le
\frac{4\varepsilon_j}{5}.
\end{aligned}
\]
Thus
$
\nu\in\mathcal F^{(4/5)}
\subset\mathcal F^{(1)}.
$
Since $\nu\in\mathcal M(X,G)$, the choice of
$\mathcal F^{(1)}$ yields
\[
\nu\in
\mathcal F^{(1)}\cap\mathcal M(X,G)
\subset
B(\mu_0,\eta_0).
\]
Hence, 
$\mathcal M(\Lambda,G)\subset B(\mu_0,\eta_0).$
\end{proof}
\subsection{Lower bound for entropy}
\label{subsec:entropy-lower}
In this subsection, we establish the entropy lower bound.
For $m\ge m_k$, let
\[
I_{k,m}
:=
\left\{
Sd\in\mathcal T_k:
Sd\subset K_m
\right\}
\]
and
\[
\widetilde K_m
:=
\bigcup_{Sd\in I_{k,m}}Sd.
\]
By the choice of $m_k$ and Lemma~\ref{lem:2.7},
$|\widetilde K_m|
>
(1-\theta)|K_m|
~~
\text{for every }m\ge m_k.$

Define
$
\Gamma(K_m)
:=
\prod_{Sd\in I_{k,m}}
\prod_{i=1}^{t_k}
\prod_{c\in C_{k,S,i}}
\Gamma_{k,S,i,c}.
$
For $\omega\in\Gamma(K_m)$, denote its coordinates by
$\omega_{S,d,i,c}\in\Gamma_{k,S,i,c},
~~
Sd\in I_{k,m}.$

Since every factor $\Gamma_{k,S,i,c}$ is nonempty, fix a
reference configuration $\omega^0\in\Omega_k$. For each
$\omega\in\Gamma(K_m)$, define
$\widehat\omega\in\Omega_k$ by
\[
\widehat{\omega}_{S,d,i,c}
:=
\begin{cases}
\omega_{S,d,i,c},
& Sd\in I_{k,m},\\[2mm]
\omega^0_{S,d,i,c},
& Sd\notin I_{k,m}.
\end{cases}
\]
Since $Z_k^\#(\widehat\omega)\neq\varnothing$, choose
$z_\omega
\in
Z_k^\#(\widehat\omega)
\subset
Z^\#.$
Then, for every $Sd\in I_{k,m}$, $1\le i\le t_k$, and
$c\in C_{k,S,i}$,
\begin{align}\label{3.14}
\rho_{J_{k,S,i,c}}
\bigl(cdz_\omega,\omega_{S,d,i,c}\bigr)
\le\Delta.
\end{align}
Set
\[Z_m^\#:=\left\{z_\omega:\omega\in\Gamma(K_m)
\right\}.\]
\begin{claim}
The set $Z_m^\#$ is $(K_m,\Delta)$-separated.
\end{claim}
\begin{proof}
Let $\omega,\omega'\in\Gamma(K_m)$ be distinct. Then there
exist $Sd\in I_{k,m}$, $1\le i\le t_k$, and
$c\in C_{k,S,i}$ such that
$\omega_{S,d,i,c}\neq\omega'_{S,d,i,c}$.
Since $\Gamma_{k,S,i,c}$ is
$\left(
J_{k,S,i,c},
\frac{\delta^*}{2},
\varepsilon^*
\right)\text{-separated}$,
there exists $u\in J_{k,S,i,c}$ such that
\[
\rho
\bigl(
u\omega_{S,d,i,c},
u\omega'_{S,d,i,c}
\bigr)
>
\varepsilon^*.
\]
By \eqref{3.14} and the triangle inequality,
\[
\rho
\bigl(
ucdz_\omega,
ucdz_{\omega'}
\bigr)
>
\varepsilon^*-2\Delta
>
\Delta.
\]
Moreover,
$ucd\in J_{k,S,i,c}cd\subset Sd\subset
K_m$.
Hence
$\rho_{K_m}(z_\omega,z_{\omega'})>
\Delta$.
Thus the map $\omega\mapsto z_\omega$ is injective, and
$Z_m^\#$ is $(K_m,\Delta)$-separated.
\end{proof}
Therefore,
$|Z_m^\#|=|\Gamma(K_m)|.$
Since
$
Z_m^\#
\subset
Z^\#
\subset
\Lambda,
$
we have
$s(\Lambda,K_m,\Delta)
\ge
|\Gamma(K_m)|$.
Using \eqref{3.4}, we obtain
\[
\begin{aligned}
|\Gamma(K_m)|
&=
\prod_{Sd\in I_{k,m}}
\prod_{i=1}^{t_k}
\prod_{c\in C_{k,S,i}}
|\Gamma_{k,S,i,c}|\\
&\ge
\exp\left(
h_1
\sum_{Sd\in I_{k,m}}
\sum_{i=1}^{t_k}
\sum_{c\in C_{k,S,i}}
|J_{k,S,i,c}|
\right).
\end{aligned}
\]
For every $Sd\in I_{k,m}$, properties~\textup{(2)} and
\textup{(3)} give
\[
\sum_{i=1}^{t_k}
\sum_{c\in C_{k,S,i}}
|J_{k,S,i,c}|
>
(1-4\alpha_k)|S|.
\]
Since $\widetilde K_m$ is the disjoint union of the complete
tiles in $I_{k,m}$,
\[
\begin{aligned}
\sum_{Sd\in I_{k,m}}
\sum_{i=1}^{t_k}
\sum_{c\in C_{k,S,i}}
|J_{k,S,i,c}|
&>
(1-4\alpha_k)|\widetilde K_m|>
(1-\theta)(1-4\alpha_k)|K_m|.
\end{aligned}
\]
Hence
$$
s(\Lambda,K_m,\Delta)
\ge
\exp\left(
h_1(1-\theta)(1-4\alpha_k)|K_m|
\right).
$$
Therefore,
\[
\begin{aligned}
h(\Lambda,G)
&\ge h(\Lambda,G,\Delta)\\
&=
\limsup_{m\to\infty}
\frac{1}{|K_m|}
\log s(\Lambda,K_m,\Delta)\\
&\ge
h_1(1-\theta)(1-4\alpha_k)\\
&>
h_0.
\end{aligned}
\]
\subsection{Upper bound for entropy}\label{upper bound 3.3}
In this subsection, we establish the entropy upper bound. 

Let
$A:=GZ^\#$.
Fix an enumeration
$G=\{\ell_1,\ell_2,\dots\}$.
For $x\in A$, define
\[
j(x):=\min\{j:\ell_j^{-1}x\in Z^\#\},
\qquad
g_x:=\ell_{j(x)},
\qquad
z_x:=g_x^{-1}x.
\]
Then
$x=g_xz_x,
~~
z_x\in Z^\#$.
Define the tiling alphabet
$\mathcal A_{\mathcal T}:=\mathcal S_k\sqcup\{0\}$
and the tiling symbolic point
$x_{\mathcal T_k}\in\mathcal A_{\mathcal T}^G$
by
$$x_{\mathcal T_k}(d)=\begin{cases}S,& \text{if }Sd\in\mathcal T_k,\\0,& \text{otherwise}.\end{cases}$$
Define the phase alphabet
\[
\mathcal A_{\mathrm{ph}}
:=
\left(
\bigcup_{S\in\mathcal S_k}
\bigcup_{i=1}^{t_k}
\bigcup_{c\in C_{k,S,i}}
\bigl(
\{S\}\times\{i\}\times\{c\}\times J_{k,S,i,c}
\bigr)
\right)
\sqcup\{\mathrm{bad}\}.
\]
Since all the involved families are finite,
$\mathcal A_{\mathrm{ph}}$ is a finite alphabet.

Fix $m\ge m_k$ sufficiently large. For $g\in G$, define
$$\operatorname{Ph}_m(g):K_m\to\mathcal A_{\mathrm{ph}}$$
by
$$\operatorname{Ph}_m(g)(s)
=
\begin{cases}
(S,i,c,t),
&
\begin{array}{l}
\text{if }S\in\mathcal S_k,\ 1\le i\le t_k,\
c\in C_{k,S,i},\\
\text{and there exists }d\in G\text{ such that }
Sd\in\mathcal T_k,\ Sd\subset K_mg,\\
sg=tcd,\quad t\in J_{k,S,i,c};
\end{array}\\[4mm]
\mathrm{bad},&\text{otherwise}.
\end{cases}$$
This is well defined. Indeed, each $sg$ belongs to a unique
tile of $\mathcal T_k$, and the injectivity of the shape-center
map determines the corresponding pair $(S,d)$ uniquely.
Within this tile, the good subtiles are pairwise disjoint, so
$(i,c)$ is unique. Finally, for fixed $c$ and $d$, the identity
$sg=tcd$ determines $t$ uniquely.

For $x=g_xz_x\in A$, define
$\operatorname{Ph}_m(x):=\operatorname{Ph}_m(g_x)$.
Choose a finite set $R\subset G$ such that
$R\supset\{e\}\cup\bigcup_{S\in\mathcal S_k}S^{-1}$.
For $g\in G$, define
$$P_m(g):=(g x_{\mathcal T_k})|_{RK_m}\in\mathcal A_{\mathcal T}^{RK_m}.$$
Here we use the right-shift convention:
$$(g x)(u)=x(ug).$$
Thus for $u\in {RK_m}$, we have $ P_m(g)(u)=x_{\mathcal T_k}(ug)$.

Define
\[
\Phi_m:
\mathcal A_{\mathcal T}^{RK_m}
\longrightarrow
\mathcal A_{\mathrm{ph}}^{K_m}
\]
as follows. Let $P\in\mathcal A_{\mathcal T}^{RK_m}$ and
$s\in K_m$. Since
$a^{-1}s\in S^{-1}K_m\subset RK_m$
for every $S\in\mathcal S_k$ and $a\in S$, the value
$P(a^{-1}s)$ is well defined.

If there exist $S,a,i,c,t$ satisfying
\[
\{(S,a)\}
=
\left\{
(S',a'):
S'\in\mathcal S_k,\ 
a'\in S',\
P(a'^{-1}s)=S'
\right\},
\]
\[
Sa^{-1}s\subset K_m,
\]
and
\[
\{(i,c,t)\}
=
\left\{
(i',c',t'):
\begin{array}{l}
1\le i'\le t_k,\quad
c'\in C_{k,S,i'},\\
t'\in J_{k,S,i',c'},\quad
a=t'c'
\end{array}
\right\},
\]
then set
\[
\Phi_m(P)(s):=(S,i,c,t).
\]
In all other cases, set
\[
\Phi_m(P)(s):=\mathrm{bad}.
\]
\begin{lemma}\label{lemma:3.4}
For every $g\in G$,
\[
\operatorname{Ph}_m(g)
=
\Phi_m(P_m(g)).
\]
\end{lemma}
\begin{proof}
Fix $g\in G$, and let $s\in K_m$. Since $\mathcal T_k$ is a
tiling, there exists a unique tile $Sd\in\mathcal T_k$
containing $sg$. Since $sg\in Sd$, there exists a unique
$a\in S$ such that
$sg=ad$.
Hence
\[
d=a^{-1}sg=(a^{-1}s)g.
\]
Since $a^{-1}\in S^{-1}\subset R$, we have
$a^{-1}s\in RK_m$, and therefore
\[
P_m(g)(a^{-1}s)
=
x_{\mathcal T_k}((a^{-1}s)g)
=
x_{\mathcal T_k}(d)
=
S.
\]
We next show that $(S,a)$ is the unique pair with this
property. Suppose that
\[
P_m(g)(a'^{-1}s)=S'
\]
for some $S'\in\mathcal S_k$ and $a'\in S'$. Set
\[
d':=(a'^{-1}s)g.
\]
Then $x_{\mathcal T_k}(d')=S'$, so
$S'd'\in\mathcal T_k$. Moreover,
$sg=a'd'\in S'd'$.
Since $Sd$ is the unique tile containing $sg$, we have
\[
S'd'=Sd.
\]
By the injectivity of the shape-center map,
\[
S'=S
\qquad\text{and}\qquad
d'=d.
\]
Since $sg=a'd'=ad$, it follows that $a'=a$. Thus $(S,a)$ is the unique pair satisfying the first condition
in the definition of $\Phi_m(P_m(g))(s)$.

Next, since $d=(a^{-1}s)g$, the condition
$Sd\subset K_mg$
in the definition of $\operatorname{Ph}_m(g)(s)$ is equivalent
to the condition
$Sa^{-1}s\subset K_m$
in the definition of $\Phi_m(P_m(g))(s)$.

Assume that the tile $Sd$ is complete. Since $sg=ad$, we have
$sg\in J_{k,S,i,c}cd$
precisely when
$a\in J_{k,S,i,c}c,$
or equivalently, when
\[
a=tc
\qquad
\text{for some }t\in J_{k,S,i,c}.
\]
The pairwise disjointness of the sets $J_{k,S,i,c}c$ ensures
that $(i,c)$ is unique, and then $t=ac^{-1}$ is also unique.
Hence, whenever $sg$ lies in a good subtile, both maps assign
the same symbol $(S,i,c,t)$ to $s$.

If $Sd$ is not complete, or if $Sd$ is complete but $sg$ lies
outside all its good subtiles, then both maps assign
$\mathrm{bad}$ to $s$.
Therefore,
\[
\operatorname{Ph}_m(g)(s)
=
\Phi_m(P_m(g))(s).
\]
Since this holds for every $s\in K_m$, we conclude that
\[
\operatorname{Ph}_m(g)
=
\Phi_m(P_m(g)).
\]
\end{proof}
Consequently, by Lemma \ref{lemma:3.4}, we have
$$\#\{\operatorname{Ph}_m(g):g\in G\}\le\#\{P_m(g):g\in G\}.$$
Let
$$X_{\mathcal T_k}=\overline{\{g x_{\mathcal T_k}:g\in G\}}\subset \mathcal A_{\mathcal T}^G.$$
For finite $F\subset G$, put
$$
N_{\mathcal T_k}(F)
=
\left|
\{y|_F:y\in X_{\mathcal T_k}\}
\right|.
$$
Since
$$P_m(g)=(g x_{\mathcal T_k})|_{RK_m},$$
we have
$$\#\{P_m(g):g\in G\}\le N_{\mathcal T_k}(RK_m).$$
Since $R$ is finite and $\{K_m\}_{m=1}^\infty$ is a Følner sequence, we have
$\frac{|RK_m|}{|K_m|}\to1$ and
$\{RK_m\}_{m=1}^\infty$ is also a Følner sequence. Hence, by 
$h(\mathcal T_k)=0$, we have
\[
\lim_{m\to\infty}\frac1{|RK_m|}
\log N_{\mathcal T_k}(RK_m)=0.
\]
Together with $\frac{|RK_m|}{|K_m|}\to1$, this implies that, for all sufficiently large $m$,
\[
\begin{aligned}
\#\{\operatorname{Ph}_m(g):g\in G\}
&\le
\#\{P_m(g):g\in G\}\\
&\le
N_{\mathcal T_k}(RK_m)\le
e^{\beta|K_m|}.
\end{aligned}
\] 
\begin{lemma}
For every sufficiently large $m$,
$$
s(A,K_m,\alpha)
\le
e^{(h_1+3\beta)|K_m|}.
$$
\end{lemma}
\begin{proof}
Fix a sufficiently large $m$, and let
$\mathcal P_m:=\{\operatorname{Ph}_m(g):g\in G\}$.
Fix $p\in\mathcal P_m$ and choose $g\in G$ such that
$
p=\operatorname{Ph}_m(g)$.
For $s\in K_m$, if 
	$p(s)=(S,i,c,t),$
	then there exists $d\in G$ such that
	$sg=tcd$.
	Let
\[
r:=dg^{-1}.
\]
Then
\[
s=tcr,
\qquad
r=c^{-1}t^{-1}s.
\]
Since a good phase requires $Sd\subset K_mg$ and $J_{k,S,i,c}cd\subset Sd$, we have
	$$
	J_{k,S,i,c}cr\subset K_m.
	$$
Define
\[
\mathscr B(p)
:=
\left\{
(S,i,c,r):
\begin{array}{l}
S\in\mathcal S_k,\quad
1\le i\le t_k,\quad
c\in C_{k,S,i},\\
\exists s\in K_m,\ \exists t\in J_{k,S,i,c}
\text{ such that}\\
p(s)=(S,i,c,t)
\text{ and }
r=c^{-1}t^{-1}s
\end{array}
\right\}.
\]
The corresponding good subtile
instances are pairwise disjoint and contained in $K_m$.
Hence
\[
\sum_{(S,i,c,r)\in\mathscr B(p)}
|J_{k,S,i,c}|
\le
|K_m|.
\]
For $x=g_xz_x\in A$, define
\[
\operatorname{Lab}_m(x)(S,i,c,r)
:=
a_{k,S,rg_x,i,c}(z_x),
\qquad
(S,i,c,r)\in
\mathscr B(\operatorname{Ph}_m(x)).
\]
 The number of good label codes
associated with $p$ is at most
\[
\prod_{(S,i,c,r)\in\mathscr B(p)}
|\Gamma_{k,S,i,c}|.
\]
By \eqref{3.4},
\[
\begin{aligned}
\prod_{(S,i,c,r)\in\mathscr B(p)}
|\Gamma_{k,S,i,c}|
&\le
\exp\left(
(h_1+\beta)
\sum_{(S,i,c,r)\in\mathscr B(p)}
|J_{k,S,i,c}|
\right)\\
&\le
e^{(h_1+\beta)|K_m|}.
\end{aligned}
\]
Define
\[
B(p):=\{s\in K_m:p(s)=\mathrm{bad}\}.
\]
Let $B_1(g)$ consist of those $s\in K_m$ for which the tile
containing $sg$ is not fully contained in $K_mg$, and let $B_2(g)$
consist of those $s\in K_m$ for which the tile containing $sg$ is
fully contained in $K_mg$, but $sg$ lies outside the good subtiles.
Then
\[
B(p)=B_1(g)\sqcup B_2(g).
\]
By the complete-tile covering estimate,
$|B_1(g)|\le\theta|K_m|$.
Since the good subtiles cover more than a proportion
$1-4\alpha_k$ of every complete tile,
$|B_2(g)|\le4\alpha_k|K_m|$.
Consequently,
\[
|B(p)|\le(\theta+4\alpha_k)|K_m|.
\]
By the choices of $\theta$ and $k$,
$(\theta+4\alpha_k)\log r_Q<\beta$.
For $x\in A$ and $s\in B(\operatorname{Ph}_m(x))$, define
	$$
	\ell(s,x)
	:=
	\min\left\{
	1\le \ell\le r_Q:
	\rho(sx,q_\ell)\le\frac{\alpha}{4}
	\right\},
	$$
	and set
	$$
	\operatorname{Bad}_m(x)(s):=q_{\ell(s,x)}.
	$$
The number of possible bad codes associated with $p$ is at most
	$$
	|Q|^{|B(p)|}
	\le
	e^{\beta |K_m|}.
	$$
	Define the total code
	$$
	\Pi_m(x)
	:=
	\left(
	\operatorname{Ph}_m(x),
	\operatorname{Lab}_m(x),
	\operatorname{Bad}_m(x)
	\right),
	$$
	and set
	$$
	\mathcal C_m:=\{\Pi_m(x):x\in A\}.
	$$
Combining the estimates for the phase, good label, and bad codes, we obtain
	$$
	|\mathcal C_m|
	\le
	e^{\beta |K_m|}
	e^{(h_1+\beta)|K_m|}
	e^{\beta |K_m|}
	=
	e^{(h_1+3\beta)|K_m|}.
	$$
	Let $E\subset A$ be any $(K_m,\alpha)$-separated set. We prove that
	$$
	\Pi_m|_E:E\to\mathcal C_m
	$$
	is injective. Suppose $x,y\in E$ and $\Pi_m(x)=\Pi_m(y)$. Write
	$$
	x=g_xz_x,
	\qquad
	y=g_yz_y.
	$$
	Fix $s\in K_m$. If $\operatorname{Ph}_m(x)(s)\ne\mathrm{bad}$, then
	$$
	\operatorname{Ph}_m(x)(s)
	=
	\operatorname{Ph}_m(y)(s)
	=
	(S,i,c,t).
	$$
	Thus there exist $d_x,d_y\in G$ such that
	$$
	sg_x=tcd_x,
	\qquad
	sg_y=tcd_y.
	$$
	Let
	$$
	r_x:=d_xg_x^{-1},
	\qquad
	r_y:=d_yg_y^{-1}.
	$$
	Then
	$$
	r_x=c^{-1}t^{-1}s=r_y.
	$$
Denote their common value by $r$. Then
$d_x=rg_x
~~and~~
d_y=rg_y$.
Since the label codes coincide at the normalized index
$(S,i,c,r)$, we have
\[
\begin{aligned}
a_{k,S,d_x,i,c}(z_x)=
\operatorname{Lab}_m(x)(S,i,c,r)=
\operatorname{Lab}_m(y)(S,i,c,r)=
a_{k,S,d_y,i,c}(z_y).
\end{aligned}
\]
Denote this common label by
$\omega\in\Gamma_{k,S,i,c}$.
Therefore
	$$
	\rho(tcd_xz_x,t\omega)\le\Delta,
	\qquad
	\rho(tcd_yz_y,t\omega)\le\Delta.
	$$
	Since
	$$
	sx=tcd_xz_x,
	\qquad
	sy=tcd_yz_y,
	$$
	we get
	$$
	\rho(sx,sy)\le2\Delta<\alpha.
	$$
If $\operatorname{Ph}_m(x)(s)=\mathrm{bad}$, then also
	$\operatorname{Ph}_m(y)(s)=\mathrm{bad}$.
	Since the bad codes coincide, there exists $q\in Q$ such that
	\[
	\rho(sx,q)\le\frac{\alpha}{4}
	\quad\text{and}\quad
	\rho(sy,q)\le\frac{\alpha}{4}.
	\]
	Therefore
	$\rho(sx,sy)\le\frac{\alpha}{2}<\alpha.$
Thus, for every $s\in K_m$,
	$\rho(sx,sy)<\alpha.$
	Hence
	$\rho_{K_m}(x,y)<\alpha$.
	Since $E$ is $(K_m,\alpha)$-separated, it follows that
$x=y$. Thus $\Pi_m|_E$ is injective.
     Hence
	$$
	|E|\le |\mathcal C_m|.
	$$
	Taking the supremum over all such $E$, we obtain
	$$
	s(A,K_m,\alpha)\le |\mathcal C_m|
	\le e^{(h_1+3\beta)|K_m|}.
	$$
\end{proof}
\begin{proposition}
Set $\gamma:=3\alpha$. Then
\[
\gamma\in(0,\varepsilon_0)
~~\text{and}~~
h(\Lambda,G,\gamma)<h_0+\beta_0.
\]
Consequently, together with the entropy lower bound,
$\Lambda$ satisfies conclusion~\textup{(2)} of
Proposition~\ref{pro:3.1}.
\end{proposition}
\begin{proof}
Fix a sufficiently large $m$, and let
$E\subset\Lambda$ be a finite
$(K_m,3\alpha)$-separated set. Since
	$\Lambda=\overline A$,
	for each $x\in E$ we can choose $y_x\in A$ such that
	$
	\rho_{K_m}(x,y_x)<\frac{\alpha}{4}.
	$
	If $x\ne x'$, then
	$$
	\rho_{K_m}(x,x')>3\alpha.
	$$
Hence
\[
\begin{aligned}
\rho_{K_m}(y_x,y_{x'})
&\ge
\rho_{K_m}(x,x')
-
\rho_{K_m}(x,y_x)
-
\rho_{K_m}(x',y_{x'})\\
&>
3\alpha-\frac{\alpha}{4}-\frac{\alpha}{4}=
\frac{5\alpha}{2}>\alpha.
\end{aligned}
\]
	Thus $\{y_x:x\in E\}$ is a $(K_m,\alpha)$-separated subset of $A$, and the map $x\mapsto y_x$ is injective. Therefore
	$$
	|E|\le s(A,K_m,\alpha).
	$$
	Since $E$ was arbitrary, we have
	$$
	s(\Lambda,K_m,3\alpha)
	\le
	s(A,K_m,\alpha)
	\le
	e^{(h_1+3\beta)|K_m|}.
	$$
By \eqref{3.1}, we have
\[
\begin{aligned}
h(\Lambda,G,\gamma)
&=
h(\Lambda,G,3\alpha)\\
&=
\limsup_{m\to\infty}
\frac{1}{|K_m|}
\log s(\Lambda,K_m,3\alpha)\\
&\le
h_1+3\beta\\
&<
h_0+\beta_0.
\end{aligned}
\]
\end{proof}
\subsection{Proof of Theorem~\ref{thm:1.3}}
In this subsection, we prove Theorem~\ref{thm:1.3}.
\begin{proof}
Let $\mu\in\mathcal M(X,G)$, and let $U$ be a neighborhood of
$\mu$ in $\mathcal M(X)$.

Suppose first that $h_\mu(X,G)=0$. Choose $\eta_0>0$ such that
$B(\mu,2\eta_0)\subset U$.
Since entropy-denseness implies that $\mathcal M_e(X,G)$ is dense
in $\mathcal M(X,G)$, there exists
$\mu_0\in\mathcal M_e(X,G)$ such that
$D(\mu,\mu_0)<\eta_0$.

If $h_{\mu_0}(X,G)>0$, Proposition~\ref{pro:3.1}, applied with
any $h_0\in(0,h_{\mu_0}(X,G))$, yields a nonempty compact
$G$-invariant set $\Lambda\subset X$ such that
$\mathcal M(\Lambda,G)\subset B(\mu_0,\eta_0)$.

If $h_{\mu_0}(X,G)=0$, the same conclusion follows by applying
Proposition~\ref{pro:3.1-zero}. 
Consequently,
$\mathcal M(\Lambda,G)\subset B(\mu,2\eta_0)\subset U$.
Thus, the result follows in the case $h_\mu(X,G)=0$.

We may therefore assume that $h_\mu(X,G)>0$. Let
$h\in(0,h_\mu(X,G))$ and fix $\varepsilon,\beta>0$. Choose
$\eta_0>0$ such that
\[
B(\mu,2\eta_0)\subset U.
\]
By Proposition~\ref{pro:2.15}, there exists
$\mu_0\in\mathcal M_e(X,G)$ such that
\[
D(\mu,\mu_0)<\eta_0
\quad\text{and}\quad
h_{\mu_0}(X,G)>h.
\]
Applying Proposition~\ref{pro:3.1} to
$\mu_0,h,\eta_0,\beta$, and $\varepsilon$, we obtain
$\gamma\in(0,\varepsilon)$ and a compact $G$-invariant subset
$\Lambda\subset X$ such that
\[
\mathcal M(\Lambda,G)\subset B(\mu_0,\eta_0),\qquad
h(\Lambda,G)>h,
\]
and
\[
h(\Lambda,G,\gamma)<h+\beta.
\]
For every $\nu\in\mathcal M(\Lambda,G)$,
\[
D(\nu,\mu)
\leq D(\nu,\mu_0)+D(\mu_0,\mu)
<2\eta_0.
\]
Hence,
\[
\mathcal M(\Lambda,G)\subset B(\mu,2\eta_0)\subset U.
\]
\end{proof}
\section{Intermediate entropies}
This section is devoted to proving Theorem~\ref{thm:1.4} and Corollary~\ref{cor:1.5}. 
Note that Theorem~\ref{thm:1.4} (\ref{theorem 4:1}) is an easy consequence of Theorem~\ref{thm:1.3} and Proposition~\ref{pro:4.1}.
\begin{proposition}\label{pro:4.1}
	Let $G$ be a countably infinite discrete amenable group and let $(X,G)$ be a compact metric $G$-system. Assume that $(X,G)$ is asymptotically entropy expansive. If $\mu\in \mathcal{M}(X,G)$ is almost entropy-approximable, then $\mu$ is entropy-approximable.
\end{proposition}
\begin{proof}
 If $h_\mu(X,G)=0$, the conclusion follows immediately from
the zero-entropy convention, since almost entropy-approximability
and entropy-approximability are defined by the same condition in
this case. Hence, we may assume that $h_\mu(X,G)>0$.

 Let $U$ be a neighborhood of $\mu$ in $\mathcal{M}(X)$, $0<h<h_\mu(X,G)~~ \text{and}~\beta>0$.
	
Since $(X,G)$ is asymptotically entropy expansive, we may choose $\varepsilon>0$ so small that
	$$
	h_{\text{loc}}(\varepsilon',G)<\frac{\beta}{2}
	\qquad
	\text{for every }0<\varepsilon'<\varepsilon.$$
By the almost entropy-approximability of $\mu$, there exist a compact $G$-invariant set $\Lambda\subset X$ and a number $\gamma\in(0,\varepsilon)$
such that
$$\mathcal{M}(\Lambda,G)\subset U,~~h(\Lambda,G)>h~~\text{and}~~h(\Lambda,G,\gamma)<h+\frac{\beta}{2}.$$
The entropy-expansiveness estimate gives
$$h(\Lambda,G)\le h(\Lambda,G,\gamma)+h_{\text{loc}}(\gamma,G).$$
Therefore
$$h(\Lambda,G)<h+\frac{\beta}{2}+\frac{\beta}{2}=h+\beta.$$
Thus $$\mathcal{M}(\Lambda,G)\subset U~~\text{and}~~h<h(\Lambda,G)<h+\beta.$$
This proves that $\mu$ is entropy-approximable.
\end{proof}
Proposition~\ref{pro:2.11} states that if the system is asymptotically entropy expansive, then the entropy map is upper semi-continuous. Thus, Theorem~\ref{thm:1.4} (\ref{theorem 4:2}) is obtained by combining Theorem~\ref{thm:1.4} (\ref{theorem 4:1}) with Proposition~\ref{pro:4.2}.
\begin{proposition}\label{pro:4.2}
	Suppose that every measure in $\mathcal M(X,G)$ is
	entropy-approximable and that the entropy map
	$\mu\mapsto h_\mu(X,G)$
	is upper semi-continuous on $\mathcal M(X,G)$. Then $(X,G)$ is
	entropy-generic.
\end{proposition}
\begin{proof}
	Fix $0\leq\alpha<h(X,G)$. Since the entropy map is upper
	semi-continuous, $\mathcal M^\alpha(X,G)$ is a compact metric
	subspace of $\mathcal M(X,G)$ and hence is a Baire space.
	For every finite real number $\alpha'>\alpha$,
	denote
	\[
	\mathcal M(\alpha,\alpha')
	:=
	\left\{
	\mu\in\mathcal M(X,G):
	\alpha\leq h_\mu(X,G)<\alpha'
	\right\}.
	\]
	By upper semi-continuity,
	$\mathcal M(0,\alpha')$
	is open in $\mathcal M(X,G)$. Since
	\[
	\mathcal M(\alpha,\alpha')
	=
	\mathcal M(0,\alpha')\cap\mathcal M^\alpha(X,G),
	\]
	the set $\mathcal M(\alpha,\alpha')$ is open in the subspace
	$\mathcal M^\alpha(X,G)$.
Let
	\[
	\mathcal M_e(\alpha,\alpha')
	:=
	\mathcal M(\alpha,\alpha')\cap\mathcal M_e(X,G).
	\]
	Since $\mathcal M_e(X,G)$ is a $G_\delta$ subset of
	$\mathcal M(X,G)$, the set $\mathcal M_e(\alpha,\alpha')$ is a
	$G_\delta$ subset of $\mathcal M^\alpha(X,G)$.
	We next prove that $\mathcal M_e(\alpha,\alpha')$ is dense in
	$\mathcal M^\alpha(X,G)$. Let
	\[
	\mu\in\mathcal M^\alpha(X,G)
	\quad\text{and}\quad
	\eta>0.
	\]
Recall that $D^*$ is the diameter of $\mathcal M(X)$. Shrinking $\eta$ if necessary, we may assume that
$0<\eta<3D^*$.
By the variational principle, we can choose
$\mu_X\in\mathcal M_e(X,G)$ such that
$h_{\mu_X}(X,G)>\alpha$.
 Denote
	\[
	\mu'
	:=
	\left(1-\frac{\eta}{3D^*}\right)\mu
	+
	\frac{\eta}{3D^*}\mu_X.
	\]
	Then
	\[
	D(\mu',\mu)\leq\frac{\eta}{3}
\quad\text{and}\quad
	h_{\mu'}(X,G)
	>
	\alpha.
	\]
Set
\[
M:=\min\{h_{\mu'}(X,G),\alpha'\}.
\]
Then $M>\alpha$. Choose
\[
a\in(\alpha,M)
\quad\text{and}\quad
0<\beta'<M-a.
\]
By the entropy-approximability of $\mu'$, there exists
a compact $G$-invariant set $\Lambda\subset X$ such that
\[
\mathcal M(\Lambda,G)
\subset B\left(\mu',\frac{\eta}{3}\right)
~~
and
~~
a<h(\Lambda,G)<a+\beta'<M.
\]
In particular,
$\alpha<h(\Lambda,G)<
\min\{h_{\mu'}(X,G),\alpha'\}$.

By the variational principle, there exists
$\nu\in\mathcal M_e(\Lambda,G)\subset B(\mu,\eta)$ such that
$h_\nu(X,G)\in[\alpha,\alpha')$.
It follows that
$\nu\in\mathcal M_e(\alpha,\alpha')$
and hence
\[
\mathcal M_e(\alpha,\alpha')\cap B(\mu,\eta)
\neq\varnothing.
\]
This implies that $\mathcal M_e(\alpha,\alpha')$ is dense in
$\mathcal M^\alpha(X,G)$.

Consequently, each $\mathcal M_e(\alpha,\alpha')$ is residual in
$\mathcal M^\alpha(X,G)$. Hence
\[
\mathcal M_e(X,G,\alpha)
=
\bigcap_{k=1}^{\infty}
\mathcal M_e\left(\alpha,\alpha+\frac{1}{k}\right)
\]
is residual in $\mathcal M^\alpha(X,G)$.
\end{proof}
Entropy-genericity yields Corollary~\ref{cor:1.5} as an immediate consequence of Theorem~\ref{thm:1.4} (\ref{theorem 4:2}) and Proposition~\ref{pro:4.3}.
\begin{proposition}\label{pro:4.3}
	Assume that $(X,G)$ is entropy-generic. 
	Then for every $\mu\in \mathcal{M}(X,G)$ and every neighborhood $U$ of $\mu$, we have
	$$
	\begin{cases}
	\mathcal{H}(X,G,U)\supset [0,h_\mu(X,G)],
	& \text{if } h_\mu(X,G)<h(X,G),\\[1mm]
	\mathcal{H}(X,G,U)\supset [0,h_\mu(X,G)),
	& \text{if } h_\mu(X,G)=h(X,G).
	\end{cases}
	$$
\end{proposition}
\begin{proof}
	There exists $\eta>0$ such that
	$B(\mu,2\eta)\subset U$.
	
	Suppose that
	\[
	0\leq\alpha<h_\mu(X,G)\leq h(X,G).
	\]
	Then $\mathcal M_e(X,G,\alpha)$ is residual in
	$\mathcal M^\alpha(X,G)$, and hence it has nonempty intersection
	with the open subset
	$B(\mu,\eta)\cap\mathcal M^\alpha(X,G)$
	of $\mathcal M^\alpha(X,G)$. Thus, there exists an ergodic
	measure
	\[
	\nu\in B(\mu,\eta)\cap\mathcal M^\alpha(X,G)\subset U
	\]
	such that
	$h_\nu(X,G)=\alpha$.
	It follows that
	\begin{align}\label{4.1}
	\mathcal H(X,G,U)\supset[0,h_\mu(X,G)).
	\end{align}
	Suppose that
	$h_\mu(X,G)<h(X,G)$.
	By the variational principle, there exists
	$\mu_0\in\mathcal M(X,G)$ such that
	$h_{\mu_0}(X,G)>h_\mu(X,G)$. In particular, $D^*>0$. Shrinking $\eta$ if necessary, we may
assume that
$0<\eta<D^*$. Denote
	\[
	\mu'
	:=
	\left(1-\frac{\eta}{D^*}\right)\mu
	+
	\frac{\eta}{D^*}\mu_0
	\in B(\mu,2\eta)\subset U.
	\]
 We still have
	$h_{\mu'}(X,G)>h_\mu(X,G)$.
Since $\mu'\in U$, applying the argument leading to
\eqref{4.1} with $\mu'$ in place of $\mu$, we obtain
	\[
	\mathcal H(X,G,U)
	\supset
	[0,h_{\mu'}(X,G))
	\supset
	[0,h_\mu(X,G)].
	\]
\end{proof}
\section*{Acknowledgments}
The second author was supported by the
National Natural Science Foundation of China (No.12471184).
The third author was supported by  Qinglan Project of Jiangsu Province of China.

\end{document}